[1994/12/01]
\documentclass{amsart}
\usepackage[all]{xy}
\usepackage{mathrsfs}
\chardef\bslash=`\\ 

\newtheorem{thm}{Theorem}[section]
\newtheorem{cor}[thm]{Corollary}
\newtheorem{lem}[thm]{Lemma}
\newtheorem{prop}[thm]{Proposition}
\newtheorem{conjecture}[thm]{Conjecture}
\newtheorem{example}[thm]{Example}

\theoremstyle{definition}
\newtheorem{defn}{Definition}[section]
\newtheorem{rem}{Remark}[section]

\theoremstyle{remark}

\newcommand{\eval}[2][\right]{\relax
  \ifx#1\right\relax \left.\fi#2#1\rvert}

\begin{document}
\title[Existence of Invariant Norms]{Existence of Invariant Norms in $p$-adic Representations of $GL_{2}(F)$ of Large Weights\thanks{Version 1.0, 2017/05/29}}
\author[E. Assaf]{Eran Assaf}
\address{Hebrew University of Jerusalem\\
Israel}
\email{eran.assaf@mail.huji.ac.il}

\begin{abstract}
In \cite{breuil2007first} Breuil and Schneider formulated a conjecture
on the equivalence of the existence of invariant norms on certain
$p$-adically locally algebraic representations of $GL_{n}(F)$ and
the existence of certain de-Rham representations of $Gal(\overline{F}/F)$,
where $F$ is a finite extension of $\mathbb{Q}_{p}$. In \cite{breuil2003quelques2,de2013existence}
Breuil and de Ieso proved that in the case $n=2$ and under some restrictions,
the existence of certain admissible filtrations on the $\phi$-module
associated to the two-dimensional de-Rham representation of $Gal(\overline{F}/F)$
implies the existence of invariant norms on the corresponding locally
algebraic representation of $GL_{2}(F)$. In \cite{breuil2003quelques2,de2013existence},
there is a significant restriction on the weight - it must be small
enough. In \cite{caraiani2013patching} the conjecture is proved
in greater generality, but the weights are still restricted to the
extended Fontaine-Laffaille range. In this paper we prove that in
the case $n=2$, even with larger weights, under some restrictions,
the existence of certain admissible filtrations implies the existence
of invariant norms.
\end{abstract}
\maketitle
\tableofcontents

\section{Introduction, Notation and Main Results}

\subsection{Introduction}

Let $p$ be a prime number. Let $F$ be a finite extension of $\mathbb{Q}_{p}$,
and let $C$ be a finite extension of $\mathbb{Q}_{p}$ which is ``large
enough'' in a precise way to be defined in Section 2. This paper
lies in the framework of the $p$-adic local Langlands programme,
whose goal is to associate to certain $n$-dimensional continuous
$p$ -adic representations of $Gal(\overline{F}/F)$, certain representations
of $G=GL_{n}(F)$.

If $F=\mathbb{Q}_{p}$ and $n=2$, then this is essentially well understood
- one has a correspondence $V\mapsto\Pi(V)$ (\cite{colmez2010representations},\cite{pavskunas2013image},\cite{colmez2013p})
associating to a $2$-dimensional $C$-representation $V$ of $Gal(\overline{\mathbb{Q}}_{p}/\mathbb{Q}_{p})$,
a unitary admissible representation of $GL_{2}(\mathbb{Q}_{p})$.
This correspondence is compatible with the classical local Langlands
correspondence and with completed étale cohomology (\cite{emerton2010local}).

Other cases seem somewhat more delicate. In particular, Breuil and
Schneider have formulated in \cite{breuil2007first} a conjecture,
generalizing a previous conjecture of Schneider and Teitelbaum \cite{schneider2006banach},
which reveals a deep connection between the category of $n$-dimensional
continuous de-Rham representations of $Gal(\overline{F}/F)$, and
certain locally algebraic representations of $GL_{n}(F)$.

By the theory of Colmez and Fontaine (\cite{colmez2000construction}),
one knows that a de-Rham representation of $Gal(\overline{F}/F)$,
$V$, is equivalent to a vector space, $D=D_{dR}(V)$, equipped with
an action of the Weil-Deligne group of $F$ and a filtration, such
that the filtration and the action satisfy a certain relation called
\emph{weak admissibility}. To this object, called the filtered $(\phi,N)$-module
attached to $V$, one can associate a smooth representation $\pi$
of $GL_{n}(F)$ by a slight modification of the classical local Langlands
correspondence (\cite{breuil2007first}, p. 16-17). On the other
hand, the Hodge-Tate weights of the filtration give rise to an irreducible
algebraic representation of $GL_{n}(F)$, which we denote by $\rho$.
The Breuil-Schneider conjecture essentially says that the existence
of a weakly admissible filtration on $D$ must be equivalent to the
existence of a $GL_{n}(F)$-invariant norm on the locally algebraic
representation $\rho\otimes\pi$. We mention that partial results,
in this generality, have been obtained by Hu (\cite{hu2009normes}),
who proved that the existence of an invariant norm on $\rho\otimes\pi$
implies the existence of a weakly admissinble filtration on $D$,
and Sorensen (\cite{sorensen2013proof}), who proved the equivalence
when $\pi$ is essentially discrete series.

In this paper we consider the particular case where $n=2$, and the
representation of the Galois group is crystalline.

Let $D$ be a $\phi$-module of rank $2$ over $F\otimes_{\mathbb{Q}_{p}}C$,
equipped with a weakly admissible filtration. Imposing some additional
technical restrictions on the weights of the filtration and on the
smooth part, we show in this paper that the locally algebraic representation
$\Pi(D)$ associated to $D$ according to the above process admits
a $G$-invariant norm. The methods we employ in order to prove this
result are well-known and were previously employed by Breuil (\cite{breuil2003quelques2})
and de Ieso (\cite{de2013existence}). The novelty of this paper
is the extension of these methods to larger weights, even though this
is accompanied by a substantial restriction on the smooth representation,
$\pi$.

We remark that in \cite{caraiani2013patching}, the authors have
proved many cases of the conjecture formulated by Breuil and Schneider,
using global methods. However, the results we obtain in this paper
are not included in their work, as they restrict the weights to be
in the extended Fontaine-Laffaille range, which, for $n=2$, means
that the weight is small.

\subsection{Notation}

Let $p$ be a prime number. Fix an algebraic closure $\overline{\mathbb{Q}}_{p}$
of $\mathbb{Q}_{p}$, and a finite extension $F$ of $\mathbb{Q}_{p}$,
contained in $\overline{\mathbb{Q}}_{p}$. Denote by $\mathcal{O}_{F}$
the ring of integers of $F$, by $\mathfrak{p}_{F}$ its maximal ideal,
and by $\kappa_{F}=\mathcal{O}_{F}/\mathfrak{p}_{F}$ its residue
field. We also fix a uniformizer $\varpi=\varpi_{F}\in\mathfrak{p}_{F}$.

Denote by $C$ a finite extension of $\mathbb{Q}_{p}$ satisfying
$\left|S\right|=[F:\mathbb{Q}_{p}]$, where $S:=Hom_{alg}(F,C)$,
and containing a square root of $\sigma(\varpi)$ for every $\sigma\in S$.

Denote by $\mathcal{O}_{C}$ the ring of integers of $C$, by $\mathfrak{p}_{C}$
its maximal ideal, and by $\kappa_{C}=\mathcal{O}_{C}/\mathfrak{p}_{C}$
its residue field. We also fix a uniformizer $\varpi=\varpi_{C}\in\mathfrak{p}_{C}$.

We denote $f=[\kappa_{F}:\mathbb{F}_{p}]$, $q=p^{f}$ the size of
the residue field, and by $e$ we denote the ramification index of
$F$ over $\mathbb{Q}_{p}$, so that $[F:\mathbb{Q}_{p}]=ef$ and
$\kappa_{F}\simeq\mathbb{F}_{q}$. We denote by $F_{0}=\mbox{Frac}(W(\kappa_{F}))$
the maximal unramified subfield of $F$, and by $\varphi_{0}$ the
absolute Frobenius of degree $p$ in $Gal(F_{0}/\mathbb{Q}_{p})$.
We denote by $Gal(\overline{F}/F)$ the Galois group of $F$ and by
$W(\overline{F}/F)$ its Weil group. Class field theory gives rise
to a homomorphism $\mbox{rec}:W(\overline{F}/F)^{ab}\rightarrow F^{\times}$
(Artin reciprocity map) which we normalize by sending the coset of
the arithmetic Frobenius to $\varpi^{-1}\mathcal{O}_{F}^{\times}$.

Denote by $v=v_{F}$ the $p$-adic valuation on $\overline{\mathbb{Q}}_{p}$
normalized by $v_{F}(\varpi)=1$. If $x\in\overline{F}$, we let $\left|x\right|=q^{-v_{F}(x)}$.
If $\lambda\in\kappa_{F}$, we denote by $[\lambda]$ the Teichmüller
representative of $\lambda$ in $\mathcal{O}_{F}$. If $\mu\in C^{\times}$,
we denote by $\mbox{nr}(\mu):F^{\times}\rightarrow C^{\times}$ the
unramified character sending $\varpi$ to $\mu$.

Denote by $\mathbf{G}$ the algebraic group $\mathbf{GL}_{2}$ defined
over $\mathcal{O}_{F}$, and let $G=\mathbf{G}(F)$ be its $F$-points.

Let $\mathbf{B}$ be the Borel subgroup of $\mathbf{G}$ consisting
of upper triangular matrices, and let $B=\mathbf{B}(F)$ be its $F$-points.

Let $\mathbf{N}$ be the unipotent radical of $\mathbf{B}$, and let
$N=\mathbf{N}(F)$ be its $F$-points.

Let $K$ be the group $GL_{2}(\mathcal{O}_{F})$, which is, up to
conjugation, the unique maximal compact subgroup of $G$. Let $I$
be the Iwahori subgroup of $K$ corresponding to $B$, and let $I(1)$
be its pro-$p$-Iwahori.

Recall that the reduction mod $\mathfrak{p}_{F}$ induces a surjective
homomorphism
\[
\mbox{red}:K\rightarrow\mathbf{G}(\kappa_{F})
\]
and that $I=\mbox{red}^{-1}(\mathbf{B}(\kappa_{F}))$ and $I(1)=\mbox{red}^{-1}(\mathbf{N}(\kappa_{F}))$.

We denote by $Z\simeq F^{\times}$ the center of $G$, and denote

\[
\alpha=\left(\begin{array}{cc}
1 & 0\\
0 & \varpi
\end{array}\right),\quad w=\left(\begin{array}{cc}
0 & 1\\
1 & 0
\end{array}\right),\quad\beta=\alpha w=\left(\begin{array}{cc}
0 & 1\\
\varpi & 0
\end{array}\right).
\]

If $\lambda\in\mathcal{O}_{F}$, we denote
\[
w_{\lambda}=\left(\begin{array}{cc}
0 & 1\\
1 & -\lambda
\end{array}\right).
\]

If $\underline{n}=(n_{\sigma})_{\sigma\in S},\underline{m}=(m_{\sigma})_{\sigma\in S}$
are elements of $\mathbb{Z}_{\ge0}^{S}$, we write:

$(i)$ $\underline{n}!=\prod_{\sigma\in S}n_{\sigma}!$

$(ii)$ $\left|\underline{n}\right|=\sum_{\sigma\in S}n_{\sigma}$

$(iii)$ $\underline{n}-\underline{m}=(n_{\sigma}-m_{\sigma})_{\sigma\in S}$

$(iv)$ $\underline{n}\le\underline{m}$ if $n_{\sigma}\le m_{\sigma}$
for all $\sigma\in S$

$(v)$ ${\underline{n} \choose \underline{m}}=\frac{\underline{n}!}{\underline{m}!(\underline{n}-\underline{m})!}$

$(vi)$ If $z\in\mathcal{O}_{F}$, we write $z^{\underline{n}}=\prod_{\sigma\in S}\sigma(z)^{n_{\sigma}}$.

\subsection{Main Results}

We fix $(\lambda_{1},\lambda_{2})\in C^{\times}\times C^{\times}$
such that $\lambda_{1}\lambda_{2}^{-1}\notin\{q^{2},1\}$ and $\underline{k}\in\mathbb{Z}_{\ge0}^{S}$.
Denote
\[
S^{+}=\{\sigma\in S\mid k_{\sigma}\ne0\}\subseteq S
\]
We also fix some $\iota\in S$, and partition $S^{+}$ according to
the action of $\sigma\in S^{+}$ on the residue field. More precisely,
for each $l\in\{0,\ldots,f-1\}$, denote
\[
J_{l}=\{\sigma\in S^{+}\mid\sigma([\zeta])=\iota\circ\varphi_{0}^{l}([\zeta])\quad\forall\zeta\in\kappa_{F}\}.
\]

For example, if $F$ is unramified, then $\left|J_{l}\right|\le1$
for all $l$.

If $i\in\mathbb{Z}$, we denote by $\overline{i}$ the unique representative
of $i\mod f$ in $\{0,\ldots,f-1\}$. For $\sigma\in J_{l}$, we denote
\[
v_{\sigma}=\inf\left\{ i\mid1\le i\le f,\quad J_{\overline{l+i}}\ne\emptyset\right\}
\]
that is, the smallest power of Frobenius $\varphi_{0}$ that is needed
to pass from $J_{l}$ to another, nonempty $J_{k}$.

We denote by $\chi:GL_{2}(F)\rightarrow F^{\times}$ the character
defined by
\[
\left(\begin{array}{cc}
a & b\\
c & d
\end{array}\right)\mapsto\varpi^{-v_{F}\left(ad-bc\right)}
\]

For $k\in\mathbb{Z}_{\ge0}$, we denote by $\rho_{k}$ the irreducible
algebraic representation of $\mathbf{G}$ of highest weight $\mbox{diag}(x_{1},x_{2})\mapsto x_{2}^{k}$
with respect to $\mathbf{B}$, the Borel subgroup of upper triangular
matrices.

We regard it also as a representation of $G=\mathbf{G}(F)$, and for
any $\sigma\in S$, denote by $\rho_{k}^{\sigma}$ the base change
of $\rho_{k}$ to a representation of $G\otimes_{F,\sigma}C$.

Also, for any $\sigma\in S$, we fix a square root of $\sigma(\varpi)$
and write $\underline{\rho}_{k}^{\sigma}=\rho_{k}^{\sigma}\otimes_{C}(\sigma\circ\chi)^{\frac{k}{2}}$.

For $\underline{k}\in\mathbb{Z}_{\ge0}^{S}$, we write

\[
\rho_{\underline{k}}=\bigotimes_{\sigma\in S}\rho_{k_{\sigma}}^{\sigma},\quad\underline{\rho}_{\underline{k}}=\bigotimes_{\sigma\in S}\underline{\rho}_{k_{\sigma}}^{\sigma}
\]

Let $\mathbf{T}$ be the standard maximal torus of $\mathbf{B}$ consisting
of diagonal matrices, and let $T=\mathbf{T}(F)$.
\begin{defn}
Let $\theta:T\rightarrow C^{\times}$ be a $C$-character of $T$
inflated to $B$, via $T\simeq B/N$. The smooth \emph{principal series
representation }corresponding to $\theta$ is
\[
Ind_{B}^{G}(\theta)=\left\{ f:G\rightarrow C\mid\begin{array}{c}
\exists U_{f}\mbox{ open\,\ s.t. }f(bgk)=\theta(b)f(g)\\
\forall g\in G,\quad b\in B,\quad k\in U_{f}
\end{array}\right\}
\]
with the group $G$ acting by right translations, namely $(gf)(x)=f(xg)$
for all $x,g\in G$ and $f\in Ind_{B}^{G}(\theta)$ .
\end{defn}
Finally, we denote by
\[
\pi=Ind_{B}^{G}(\mbox{nr}(\lambda_{1}^{-1})\otimes\mbox{nr}(\lambda_{2}^{-1}))
\]
the smooth unramified parabolic induction.

Note that the hypothesis on $(\lambda_{1},\lambda_{2})$ assures us
that $\pi$ is irreducible.

We shall from now on consider the irreducible locally algebraic representations
of the form $\underline{\rho}_{\underline{k}}\otimes\pi$.

Note that $\rho_{k}$ is not the most general irreducible algebraic
representation of $\mathbf{G}$, as it can be twisted by a power of
the determinant.

However, for the purpose of existence of $G$-invariant norms, a twist
by a power of the determinant is equivalent to a twist by a power
of $\chi$, which can be then absorbed by $\pi$ into the values of
$\lambda_{1},\lambda_{2}$.

The Breuil-Schneider conjecture can be reformulated as follows (see
\cite{de2013existence})
\begin{conjecture}
\label{conj:The-following-two}The following two statements are equivalent:

$(i)$ The representation $\rho_{\underline{k}}\otimes\pi$ admits
a $G$-invariant norm, i.e. a $p$-adic norm such that $\left\Vert gv\right\Vert =\left\Vert v\right\Vert $
for all $g\in G$ and $v\in\rho_{\underline{k}}\otimes\pi$.

$(ii)$ The following inequalities are satisfied:
\end{conjecture}
\begin{itemize}
\item $v_{F}(\lambda_{1}^{-1})+v_{F}(\lambda_{2}^{-1})+\left|\underline{k}\right|=0$
\item $v_{F}(\lambda_{2}^{-1})+\left|\underline{k}\right|\ge0$
\item $v_{F}(q\lambda_{1}^{-1})+\left|\underline{k}\right|\ge0$
\end{itemize}
The implication $(i)\Rightarrow(ii)$ of Conjecture \ref{conj:The-following-two}
follows from the work of Hu, which shows it in full generality (for
$GL_{n}(F)$) in \cite{hu2009normes}, using a result of Emerton
(\cite{emerton2005p}, Lemma 1.6).

It remains to show $(ii)\Rightarrow(i)$.

The case $\lambda_{1}\in\mathcal{O}_{C}^{\times}$ (resp. $q\lambda_{2}\in\mathcal{O}_{C}^{\times}$)
is treated in \cite[Prop. 4.10]{de2013existence} hence we may assume
that $\lambda_{1},q\lambda_{2}\notin\mathcal{O}_{C}^{\times}$.

In \cite{breuil2003quelques2,de2013existence} Breuil and de Ieso
represent $\underline{\rho}_{\underline{k}}\otimes\pi$ as a quotient
of a compact induction.

We briefly recall the definition of locally algebraic compact induction.
\begin{defn}
Let $G$ be a topological group, and let $H$ be a closed subgroup.
Let $R$ be either $\mathcal{O}_{C}$ or $C$. Let $(\pi,V)$ be an
$R$-linear representation of $H$ over a free $R$-module of finite
rank $V$. We denote by $ind_{H}^{G}\pi$ or by $ind_{H}^{G}V$ the
locally algebraic compact induction of $(\pi,V)$ from $H$ to $G$.
The space of the representation is
\begin{multline}
ind_{H}^{G}\pi= \\
\left\{ f:G\rightarrow V\mid\begin{array}{c}
f(hg)=\pi(h)f(g)\quad\forall h\in H\\
\mbox{ \ensuremath{f}\,\ has compact support mod }H,\quad f\mbox{ is locally algebraic}
\end{array}\right\}
\end{multline}
and $G$ acts on $ind_{H}^{G}\pi$ by right translation, i.e. $(gf)(x)=f(xg)$
for all $g,x\in G$.
\end{defn}
Then
\[
\underline{\rho}_{\underline{k}}\otimes\pi\simeq\frac{ind_{KZ}^{G}\underline{\rho}_{\underline{k}}}{(T-a)ind_{KZ}^{G}\underline{\rho}_{\underline{k}}}=:\Pi_{\underline{k},a}
\]
where $a=\lambda_{1}+q\lambda_{2}\in\mathfrak{p}_{C}$ , $\underline{\rho}_{\underline{k}}^{0}$
is an $\mathcal{O}_{C}$-lattice in $\underline{\rho}_{\underline{k}}$,
$ind_{KZ}^{G}$ denotes the compact induction, and $T$ is the usual
Hecke operator \cite{barthel1994irreducible}.

We then have a natural map
\[
\theta:\frac{ind_{KZ}^{G}\underline{\rho}_{\underline{k}}^{0}}{(T-a)(ind_{KZ}^{G}\underline{\rho}_{\underline{k}}^{0})}\rightarrow\Pi_{\underline{k},a}
\]
whose image is denoted by $\Theta_{\underline{k},a}$.

This is a sub-$\mathcal{O}_{C}[K]$-module of finite type which generates
$\underline{\rho}_{\underline{k}}\otimes\pi$ over $C$.

Proving Conjecture \ref{conj:The-following-two} is then equivalent
to proving that $\Theta_{\underline{k},a}$ is separated, i.e. does
not contain a $C$-line (see \cite[Prop. 1.17]{emerton2005p}) .
In this paper, we prove that this is the case, for some additional
values of $\underline{k}$ and $a$.

This generalizes the previous works of Breuil and de Ieso in \cite{breuil2003quelques2,de2013existence},
using similar methods.

In fact, de Ieso proves the following theorem:
\begin{thm}
We follow the preceding notations. The morphism $\theta$ is injective
if and only if the following two conditions are satisfied:

$(i)$ For all $l\in\{0,\ldots,f-1\},\quad\left|J_{l}\right|\le1$
.

$(ii)$ For all $\sigma\in J_{l}$
\[
k_{\sigma}+1\le p^{v_{\sigma}}.
\]
\end{thm}
As a corollary, it follows that under these conditions $\Theta_{\underline{k},a}$
is separated.

In this paper, we prove that even in some cases where $\theta$ is
not injective, the lattice $\Theta_{\underline{k},a}$ is still separated.
Namely, we prove the following theorem:
\begin{thm}
\label{thm:Main Result}We follow the preceding notations. Assume
that $\left|S^{+}\right|=1$, denote by $\sigma$ the unique element
in $S^{+}$, and let $k=k_{\sigma}=d\cdot q+r$, with $0\le r<q$.
Assume that one of the following three conditions is satisfied:

$(i)$ $k\le\frac{1}{2}q^{2}$ with $r<q-d$ and $v_{F}(a)\in[0,1]$.

$(ii)$ $k\le\frac{1}{2}q^{2}$ with $2v_{F}(a)-1\le r<q-d$ and $v_{F}(a)\in[1,e]$.

$(ii)$ $k\le\min\left(p\cdot q-1,\frac{1}{2}q^{2}\right)$ , $d-1\le r$
and $v_{F}(a)\ge d$.

Then $\Theta_{\underline{k},a}$ is separated.
\end{thm}
Therefore, these conditions on $\underline{k},a$ ensure the existence
of a $G$-invariant norm on $\underline{\rho}_{\underline{k}}\otimes\pi$,
establishing new cases of Conjecture \ref{conj:The-following-two}.
\begin{example}
Here are a couple of explicit examples for the established new cases:
\end{example}
\begin{enumerate}
\item Let $p\ne2$, $k=\frac{1}{2}(q^{2}-1)$ and $v_{F}(a)\in[0,\min(e,\frac{q+1}{4})]$.
Then, as $k=\frac{1}{2}(q-1)q+\frac{1}{2}(q-1)$, we see that $d=r=\frac{1}{2}(q-1)$,
hence
\[
2v_{F}(a)-1\le2\cdot\frac{q+1}{4}-1=\frac{1}{2}(q-1)=r<q-d=\frac{1}{2}(q+1)
\]
so either $(i)$ or $(ii)$ in Theorem \ref{thm:Main Result} is satisfied,
showing that the lattice $\Theta_{\underline{k},a}$ is separated
in this case.
\item Let $q=p\ne2$, $k=\frac{1}{2}(p^{2}-1)$ and $v_{F}(a)\ge\frac{1}{2}(p-1)$.
As in the previous example, $d=r=\frac{1}{2}(p-1)$, hence $d-1\le r$,
and $v_{F}(a)\ge d$. This shows that condition $(iii)$ in Theorem
\ref{thm:Main Result} is satisfied, showing that the lattice $\Theta_{\underline{k},a}$
is separated in this case.
\end{enumerate}

\section{Preliminaries}

\subsection{The Bruhat-Tits Tree}

We refer to \cite{breuil2003quelques1} and \cite{serre1980trees}
for further details concerning the construction and properties of
the Bruhat-Tits tree of $G$.

Let $\mathcal{T}$ be the Bruhat-Tits tree of $G$: its vertices are
in equivariant bijection with the left cosets $G/KZ$.

The tree $\mathcal{T}$ is equipped with a combinatorial distance,
and $G$ acts on it by isometries.

We denote by $s_{0}$ the \emph{standard vertex}, corresponding to
the trivial class $KZ$.

Equivalently, as the vertices are in equivariant bijection with homothety
classes of lattices in $F^{2}$, $s_{0}$ corresponds to the homothety
class of the lattice $\mathcal{O}_{F}\oplus\mathcal{O}_{F}$.

For $n\ge0$, we call the collection of vertices in $\mathcal{T}$
at distance $n$ from the standard vertex $s_{0}$, \emph{the circle
of radius $n$.}

Recall that we have the Cartan decomposition
\begin{equation}
G=\coprod_{n\in\mathbb{N}}KZ\alpha^{-n}KZ=\left(\coprod_{n\in\mathbb{N}}IZ\alpha^{-n}KZ\right)\coprod\left(\coprod_{n\in\mathbb{N}}IZ\beta\alpha^{-n}KZ\right).\label{eq:Cartan Decomposition}
\end{equation}
In particular, for any $n\in\mathbb{N}$, the classes of $KZ\alpha^{-n}KZ/KZ$
correspond to vertices $s_{i}$ of $\mathcal{T}$ such that $d(s_{i},s_{0})=n$.
Denote $I_{0}=\{0\}$, and for any $n\in\mathbb{N}_{>0}$
\[
I_{n}=\left\{ [\mu_{0}]+\varpi[\mu_{1}]+\ldots+\varpi^{n-1}[\mu_{n-1}]\mid(\mu_{0},\ldots,\mu_{n-1})\in\kappa_{F}^{n}\right\} \subseteq\mathcal{O}_{F}
\]
is a set of representatives for $\mathcal{O}_{F}/\varpi^{n}\mathcal{O}_{F}$.

For $n\in\mathbb{N}$ and $\mu\in I_{n}$, we denote :
\[
g_{n,\mu}^{0}=\left(\begin{array}{cc}
\varpi^{n} & \mu\\
0 & 1
\end{array}\right),\quad g_{n.\mu}^{1}=\left(\begin{array}{cc}
1 & 0\\
\varpi\mu & \varpi^{n+1}
\end{array}\right).
\]

We note that $g_{0,0}^{0}$ is the identity matrix, $g_{0,0}^{1}=\alpha$
and that, for all $n\in\mathbb{N}$ and any $\mu\in I_{n}$ , we have
$g_{n,\mu}^{1}=\beta g_{n,\mu}^{0}w$. Then, $g_{n,\mu}^{0}$ and
$g_{n,\mu}^{1}$ define a system of representatives for $G/KZ$:
\begin{equation}
G=\left(\coprod_{n\in\mathbb{N},\mu\in I_{n}}g_{n,\mu}^{0}KZ\right)\coprod\left(\coprod_{n\in\mathbb{N},\mu\in I_{n}}g_{n,\mu}^{1}KZ\right).\label{eq:system of representatives for G/KZ}
\end{equation}

For $n\in\mathbb{N}$ we denote
\[
S_{n}^{0}=IZ\alpha^{-n}KZ=\coprod_{\mu\in I_{n}}g_{n,\mu}^{0}KZ,\quad S_{n}^{1}=IZ\beta\alpha^{-n}KZ=\coprod_{\mu\in I_{n}}g_{n,\mu}^{1}KZ
\]
and we let $S_{n}=S_{n}^{0}\coprod S_{n}^{1}$ and $B_{n}=B_{n}^{0}\coprod B_{n}^{1}$,
where $B_{n}^{0}=\coprod_{m\le n}S_{m}^{0}$ and $B_{n}^{1}=\coprod_{m\le n}S_{m}^{1}$.

In particular, we have $S_{0}=KZ\coprod\alpha KZ$.
\begin{rem}
Recall, as in \cite{breuil2003quelques1,de2013existence} that $S_{n}^{0}\coprod S_{n-1}^{1}$
(resp. $B_{n}^{0}\coprod B_{n-1}^{1}$) is the collection of vertices
in $\mathcal{T}$ at distance $n$ (resp. at most $n$) from $s_{0}$.
Similarly, $S_{n}^{1}\coprod S_{n-1}^{0}$ (resp. $B_{n}^{1}\coprod B_{n-1}^{0}$)
is the collection of vertices in $\mathcal{T}$ at distance $n$ (resp.
at most $n$) from $\alpha s_{0}$.
\end{rem}
We denote by $R$ either the field $C$ or its ring of integers $\mathcal{O}_{C}$.
Let $\sigma$ be a continuous $R$-linear representation of $KZ$
on a free $R$-module of finite rank $V_{\sigma}$. We denote by $ind_{KZ}^{G}\sigma$
the $R$-module of functions $f:G\rightarrow V_{\sigma}$ compactly
supported modulo $Z$, such that
\[
f(\kappa g)=\sigma(\kappa)f(g)\quad\forall\kappa\in KZ,g\in G
\]

with $G$ acting by right translations, i.e. $(g\cdot f)(g')=f(g'g)$.

As in \cite{barthel1994irreducible}, for $g\in G$, $v\in V_{\sigma}$,
we denote by $[g,v]$ the element of $ind_{KZ}^{G}\sigma$ supported
on $KZg^{-1}$ and such that $[g,v](g^{-1})=v$.

Then we have
\begin{equation}
\begin{split}
\forall g,g'\in G,v\in V_{\sigma}\quad g\cdot[g',v]=[gg',v] \\
\forall g\in G,\kappa\in KZ,v\in V_{\sigma}\quad[g\kappa,v]=[g,\sigma(\kappa)v]
\end{split}
\end{equation}

We can think of $ind_{KZ}^{G}\sigma$ as a vertex coefficient system
on $\mathcal{T}$, having $\sigma$ as the module on each vertex,
identifying $[g,v]$ with the vector $v$ at the vertex corresponding
to $g$, i.e. identifying vertex $g$ with $KZg^{-1}$. Note that
the choice of representative for $gKZ$ affects the choice of vector
$v\in\sigma$.

Recall the following result (\cite[$\mathcal{x}$2]{barthel1994irreducible}),
which gives a basis for the $R[G]$-module $ind_{KZ}^{G}\sigma$.
\begin{prop}
\label{prop:basis for compact induction}Let $\mathcal{B}$ be a basis
for $V_{\sigma}$ over $R$, and let $\mathcal{G}$ be a system of
representatives for left cosets of $G/KZ$. Then the family of functions
$\mathcal{I}:=\left\{ [g,v]\mid g\in\mathcal{G},v\in\mathcal{B}\right\} $
forms a basis for $ind_{KZ}^{G}\sigma$ over $R$.
\end{prop}
\begin{rem}
\label{rem:induction as tensor product}The representation $ind_{KZ}^{G}\sigma$
is isomorphic to the representation of $G$ given by the $R[G]$-module
$R[G]\otimes_{R[KZ]}V_{\sigma}$. More precisely, if $g\in G$ and
$v\in V_{\sigma}$ , then the element $g\otimes v$ corresponds to
the function $[g,v]$.

From proposition \ref{prop:basis for compact induction} and the decomposition
(\ref{eq:system of representatives for G/KZ}), any function $f\in ind_{KZ}^{G}\sigma$
can be written uniquely as a finite sum of the form
\[
f=\sum_{n=0}^{n_{0}}\sum_{\mu\in I_{n}}\left([g_{n,\mu}^{0},v_{n,\mu}^{0}]+[g_{n,\mu}^{1},v_{n,\mu}^{1}]\right)
\]
with $v_{n.\mu}^{0},v_{n,\mu}^{1}\in V_{\sigma}$, and where $n_{0}$
is a non-negative integer, which depends on $f$. We call the \emph{support}
of $f$ the collection of $g_{n,\mu}^{i}$ such that $v_{n,\mu}^{i}\ne0$.
We write $f\in S_{n}$ (resp. $B_{n},S_{n}^{0}$, etc. ) if the support
of $f$ is contained in $S_{n}$ (resp. $B_{n},S_{n}^{0}$, etc. ).
We write $f\in B^{0}$ if the support of $f$ is contained in $B_{n}^{0}$
for some $n$, and $f\in B^{1}$ if the support of $f$ is contained
in $B_{n}^{1}$ for some $n$.

Let $\pi$ be a continuous $R$-linear representation of $G$ over
an $R$-module. From \cite{barthel1994irreducible}, we have a canonical
isomorphism of $R$-modules
\[
Hom_{R[G]}(ind_{KZ}^{G}\sigma,\pi)\simeq Hom_{R[KZ]}(\sigma,\pi\mid_{KZ})
\]
which translates to the fact that the functor of compact induction
$ind_{KZ}^{G}$ is left adjoint to the restriction functor, and is
called \emph{compact Frobenius reciprocity}.
\end{rem}

\subsection{Hecke Algebras}

Let $\sigma$ be a continuous $R$-linear representation of $KZ$
over a free $R$-module $V_{\sigma}$ of finite rank. The Hecke algebra
$\mathcal{H}(KZ,\sigma)$ associated to $KZ$ and $\sigma$ is the
$R$-algebra defined by
\[
\mathcal{H}(KZ,\sigma)=End_{R[G]}(ind_{KZ}^{G}\sigma).
\]
We can interpret $\mathcal{H}(KZ,\sigma)$ as a convolution algebra.
In fact, denote by $\mathcal{H}_{KZ}(\sigma)$ the $R$-module of
functions $\varphi:G\rightarrow End_{R}(V_{\sigma})$ compactly supported
modulo $Z$, such that
\[
\forall\kappa_{1},\kappa_{2}\in KZ,\quad\forall g\in G,\quad\varphi(\kappa_{1}g\kappa_{2})=\sigma(\kappa_{1})\circ\varphi(g)\circ\sigma(\kappa_{2}).
\]
This is a unitary $R$-algebra with the convolution product defined,
for all $\varphi_{1},\varphi_{2}\in\mathcal{H}_{KZ}(\sigma)$ and
all $g\in G$, by the following formula:
\[
\varphi_{1}*\varphi_{2}(g)=\sum_{xKZ\in G/KZ}\varphi_{1}(x)\circ\varphi_{2}(x^{-1}g).
\]

It admits as a unit element the function $\varphi_{e}=[1,id]$ defined
by
\[
\varphi_{e}(g)=\begin{cases}
\sigma(g) & g\in KZ\\
0 & else
\end{cases}.
\]
One may verify that the bilinear map
\begin{eqnarray*}
\mathcal{H}_{KZ}(\sigma)\times ind_{KZ}^{G}\sigma & \rightarrow & ind_{KZ}^{G}\sigma\\
(\varphi,f) & \mapsto & T_{\varphi}(f)(g):=\sum_{xKZ\in G/KZ}\varphi(x)\left(f(x^{-1}g)\right)
\end{eqnarray*}
equips $ind_{KZ}^{G}\sigma$ with the structure of a left $\mathcal{H}_{KZ}(\sigma)$-module,
which commutes with the action of $G$.

The following Lemma is well known, see e.g. \cite[Lemma 2.4]{de2013existence}.
\begin{lem}
\label{lem:hecke algebra is convolution}The map
\begin{eqnarray*}
\mathcal{H}_{KZ}(\sigma) & \rightarrow & \mathcal{H}(KZ,\sigma)\\
\varphi & \mapsto & T_{\varphi}(f)
\end{eqnarray*}
is an isomorphism of $R$-algebras. In particular, if $g\in G$, and
if $v\in V_{\sigma}$, the action of $T_{\varphi}$ on $[g,v]$ is
given by
\begin{equation}
T_{\varphi}([g,v])=\sum_{xKZ\in G/KZ}[gx,\varphi(x^{-1})(v)].\label{eq:T action Frobenius reciprocity}
\end{equation}
\end{lem}
We assume now that $R=C$. Denote by $\mathbf{1}$ the trivial representation
of $KZ$ and assume that $\sigma$ is the restriction to $KZ$ of
a locally analytic representation (in the sense of \cite{schneider2003algebras,schneider2002locally})
of $G$ on $V_{\sigma}$. By \cite{schneider2006banach}, the map
\begin{eqnarray*}
\iota_{\sigma}:\mathcal{H}_{KZ}(\mathbf{1}) & \rightarrow & \mathcal{H}_{KZ}(\sigma)\\
\varphi & \mapsto & (\varphi\cdot\sigma)(g):=\varphi(g)\sigma(g)
\end{eqnarray*}
is then an injective homomorphism of $C$-algebras. Before we state
a condition assuring the bijectivity of $\iota_{\sigma}$, we recall
the existence of a $\mathbb{Q}_{p}$-linear action of the Lie algebra
$\mathfrak{g}$ of $G$ on the space $V_{\sigma}$ defined by
\[
\forall\mathfrak{x}\in\mathfrak{g},\forall v\in V_{\sigma},\quad\mathfrak{x}v=\frac{d}{dt}\mbox{exp}(t\mathfrak{x})v\mid_{t=0}
\]
where $\mbox{exp}:\mathfrak{g}\dasharrow G$ denotes the exponential
map defined locally in the neighbourhood of $0$ (\cite[$\mathcal{x}$2]{schneider2002locally}).

This action is extended to an action of the Lie algebra $\mathfrak{g}\otimes_{\mathbb{Q}_{p}}C$,
and allows de Ieso to obtain the following result: (see \cite[Lemma 4.2.5]{de2013existence})
\begin{lem}
\label{lem:isomorphism of Hecke algebras in the abs irred case}If
the $\mathfrak{g}\otimes_{\mathbb{Q}_{p}}C$-module $V_{\sigma}$
is absolutely irreducible, then the map $\iota_{\sigma}$ is bijective.
\end{lem}

\section{Representations of $GL_{2}(F)$}

\subsection{\label{subsec:-algebraic-representations-of G}$\mathbb{Q}_{p}$-algebraic
representations of $GL_{2}(F)$}

For $k\in\mathbb{N}$, we denote by $\rho_{k}$ the irreducible algebraic
representation of $\mathbf{G}$ of highest weight $\mbox{diag}(x_{1},x_{2})\mapsto x_{2}^{k}$
with respect to $\mathbf{B}$, and we consider it also as a representation
of $G=\mathbf{G}(F)$.

For $\sigma\in S$, we denote by $\rho_{k}^{\sigma}$ the base change
of $\rho_{k}$ to a representation of $G\otimes_{F,\sigma}C$.

We denote by $\chi:GL_{2}(F)\rightarrow F^{\times}$ the character
defined by
\[
\left(\begin{array}{cc}
a & b\\
c & d
\end{array}\right)\mapsto\varpi^{-v_{F}(ad-bc)}.
\]

Also, choose a square root of $\sigma(\pi)$ in $C$, and let
\[
\underline{\rho}_{k}^{\sigma}=\rho_{k}^{\sigma}\otimes_{C}(\sigma\circ\chi)^{\frac{k}{2}}.
\]

For $\sigma\in S$ and $k\in\mathbb{N}$, we identify $\underline{\rho}_{k}^{\sigma}$
with the representation of $G$ given by the $C$-vector space
\[
\bigoplus_{i=0}^{k}C\cdot x_{\sigma}^{k-i}y_{\sigma}^{i}
\]
of homogeneous polynomials of degree $k$ in $x_{\sigma},y_{\sigma}$
with coefficients in $C$, on which $G$ acts by the following formula:
\begin{equation}
\begin{split}
&\left(\begin{array}{cc}
a & b\\
c & d
\end{array}\right)(x_{\sigma}^{k-i}y_{\sigma}^{i})=\\
&=\left(\sigma\circ\chi\left(\begin{array}{cc}
a & b\\
c & d
\end{array}\right)\right)^{\frac{k}{2}}(\sigma(a)x_{\sigma}+\sigma(c)y_{\sigma})^{k-i}(\sigma(b)x_{\sigma}+\sigma(d)y_{\sigma})^{i}.\label{eq:formula for G algebraic action}
\end{split}
\end{equation}

If $w_{\sigma}\in\underline{\rho}_{k}^{\sigma}$ and if $g\in G$,
we denote simply $gw_{\sigma}$ for the vector obtained from letting
$g$ act on $w_{\sigma}$.
\begin{rem}
The formula (\ref{eq:formula for G algebraic action}) assures, in
particular, that for every $w_{\sigma}\in\underline{\rho}_{k}^{\sigma}$
\[
\left(\begin{array}{cc}
\varpi & 0\\
0 & \varpi
\end{array}\right)w_{\sigma}=w_{\sigma}.
\]

Fix $\underline{k}=(k_{\sigma})_{\sigma\in S}\in\mathbb{N}^{S}$,
and let
\[
I_{\underline{k}}=\left\{ \underline{i}=(i_{\sigma})_{\sigma\in S}\in\mathbb{N}^{S},\quad0\le i_{\sigma}\le k_{\sigma}\quad\forall\sigma\in S\right\} .
\]

We denote by $\rho_{\underline{k}}$ (resp. $\underline{\rho}_{\underline{k}}$
) the representation of $G$ on the following vector space
\[
V_{\rho_{\underline{k}}}:=\bigotimes_{\sigma\in S}\rho_{k_{\sigma}}^{\sigma}\quad\left(\mbox{resp.}\quad V_{\underline{\rho}_{\underline{k}}}:=\bigotimes_{\sigma\in S}\underline{\rho}_{k_{\sigma}}^{\sigma}\right)
\]

on which an element $\left(\begin{array}{cc}
a & b\\
c & d
\end{array}\right)\in G$ acts componentwise. In particular, for all $\bigotimes_{\sigma\in S}w_{\sigma}\in V_{\underline{\rho}_{\underline{k}}}$
we have:
\[
\underline{\rho}_{\underline{k}}\left(\begin{array}{cc}
a & b\\
c & d
\end{array}\right)\left(\bigotimes_{\sigma\in S}w_{\sigma}\right)=\bigotimes_{\sigma\in S}\left(\left(\begin{array}{cc}
a & b\\
c & d
\end{array}\right)w_{\sigma}\right).
\]

These are two absolutely irreducible representations of $G$ which
remain absolutely irreducible even when we restrict them to the action
of an open subgroup of $G$ (\cite[$\mathcal{x}$2]{breuil2007first}).

For all $\underline{i}\in I_{\underline{k}}$, we let:
\[
e_{\underline{k},\underline{i}}:=\bigotimes_{\sigma\in S}e_{k_{\sigma},i_{\sigma}}
\]
where, for any $\sigma\in S$, $e_{k_{\sigma},i_{\sigma}}$ denotes
the monomial $x_{\sigma}^{k_{\sigma}-i_{\sigma}}y_{\sigma}^{i_{\sigma}}$.
We then denote by $U_{\underline{k}}$ the endomorphism of $V_{\underline{\rho}_{\underline{k}}}$
defined by
\[
U_{\underline{k}}:=\bigotimes_{\sigma\in S}U_{k_{\sigma}}^{\sigma}
\]
where $U_{k}^{\sigma}$ denotes, for all $\sigma\in S$ and $k\in\mathbb{N}$,
the endomorphism of $\underline{\rho}_{k}^{\sigma}$ given, with respect
to the basis $\left(e_{k,i}\right)_{i=0}^{k}$ by the diagonal matrix
\[
U_{k}^{\sigma}=\left(\begin{array}{cccc}
\sigma(\varpi)^{k} & 0 & \cdots & 0\\
0 & \sigma(\varpi)^{k-1} & \ddots & \vdots\\
\vdots & \ddots & \ddots & 0\\
0 & \cdots & 0 & 1
\end{array}\right).
\]
In \cite[Lemma 3.2]{de2013existence}, de Ieso proves the following
Lemma.
\end{rem}
\begin{lem}
\label{lem:existence of the Hecke operator}There exists a unique
function $\psi:G\rightarrow End_{C}(V_{\underline{\rho}_{\underline{k}}})$
supported in $KZ\alpha^{-1}KZ$ such that:

$(i)$ For any $\kappa_{1},\kappa_{2}\in KZ$ we have $\psi(\kappa_{1}\alpha^{-1}\kappa_{2})=\underline{\rho}_{\underline{k}}(\kappa_{1})\circ\psi(\alpha^{-1})\circ\underline{\rho}_{\underline{k}}(\kappa_{2})$.

$(ii)$ $\psi(\alpha^{-1})=U_{\underline{k}}$.
\end{lem}
We remark that in fact, $\psi=\underline{\rho}_{\underline{k}}\mid_{KZ\alpha^{-1}KZ}$,
since
\begin{equation}
U_{\underline{k}}=\underline{\rho}_{\underline{k}}(\alpha^{-1})\label{eq:U is rho}
\end{equation}

By Lemma \ref{lem:hecke algebra is convolution}, we know that the
Hecke algebra $\mathcal{H}(KZ,\underline{\rho}_{\underline{k}})$
is naturally isomorphic to the convolution algebra $\mathcal{H}_{KZ}(\underline{\rho}_{\underline{k}})$
of functions $\varphi:G\rightarrow End_{C}(V_{\underline{\rho}_{\underline{k}}})$
compactly supported modulo $Z$, such that
\[
\forall\kappa_{1},\kappa_{2}\in KZ,g\in G,\quad\varphi(\kappa_{1}g\kappa_{2})=\underline{\rho}_{\underline{k}}(\kappa_{1})\circ\varphi(g)\circ\underline{\rho}_{\underline{k}}(\kappa_{2}).
\]

It follows that the map $\psi$ from Lemma \ref{lem:existence of the Hecke operator}
corresponds to an operator $T\in\mathcal{H}(KZ,\underline{\rho}_{\underline{k}})$
whose action on the elements $[g,v]$ for $g\in G$ and $v\in V_{\underline{\rho}_{\underline{k}}}$
is given by the formula (\ref{eq:T action Frobenius reciprocity}).

Moreover,
\begin{rem}
A simple argument using the Bruhat-Tits tree of $G$ shows that $T$
is injective on $ind_{KZ}^{G}\underline{\rho}_{\underline{k}}$.
\end{rem}

\subsection{\label{subsec:Lattices}Lattices}

We keep the notations of Section \ref{subsec:-algebraic-representations-of G}
and denote by $\underline{\rho}_{k}^{\sigma,0}$, for $\sigma\in S$
and $k\in\mathbb{N}$, the representation of the group $KZ$ on the
$\mathcal{O}_{C}$-module
\[
\bigoplus_{i=0}^{k}\mathcal{O}_{C}\cdot x_{\sigma}^{k-i}y_{\sigma}^{i}
\]
 of homogeneous polynomials of degree $k$, on which an element $\left(\begin{array}{cc}
a & b\\
c & d
\end{array}\right)\in K$ acts by
\[
\left(\begin{array}{cc}
a & b\\
c & d
\end{array}\right)\left(x_{\sigma}^{k-i}y_{\sigma}^{i}\right)=\left(\sigma(a)x_{\sigma}+\sigma(c)y_{\sigma}\right)^{k-i}\left(\sigma(b)x_{\sigma}+\sigma(d)y_{\sigma}\right)^{i}
\]
and the matrix $\left(\begin{array}{cc}
\varpi & 0\\
0 & \varpi
\end{array}\right)\in Z$ acts as the identity. If $w_{\sigma}\in\underline{\rho}_{k}^{\sigma,0}$
and if $g\in G$, we simply denote by $gw_{\sigma}$ the vector obtained
from letting $g$ act on $w_{\sigma}$.
\begin{defn}
Let $V$ be a $C$-vector space. A \emph{lattice }$\mathcal{L}$ in
$V$ is a sub-$\mathcal{O}_{C}$-module of $V$, such that, for any
$v\in V$, there exists a nonzero element $a\in C^{\times}$ such
that $av\in\mathcal{L}$. A lattice $\mathcal{L}$ is called \emph{separated}
if $\bigcap_{n\in\mathbb{N}}\varpi^{n}\mathcal{L}=0$, which is equivalent
to demanding that it contains no $C$-line.
\end{defn}
\begin{example}
\label{exa:separated lattice}The $\mathcal{O}_{C}$-module $\underline{\rho}_{k}^{\sigma,0}$
is a separated lattice of $\underline{\rho}_{k}^{\sigma}$, which
is moreover stable under the action of $KZ$.
\end{example}
\begin{rem}
There are many choices of possible separated lattices in $\underline{\rho}_{k}^{\sigma}$,
which are stable under the action of $KZ$. Another natural choice
(and in some sense even more natural than ours), as pointed out by
C. Breuil, is the lattice
\[
\bigoplus_{i=0}^{k}\mathcal{O}_{C}\cdot\frac{x_{\sigma}^{k-i}y_{\sigma}^{i}}{(k-i)!\cdot i!}
\]
which, in the case $q>p$, is different from $\underline{\rho}_{k}^{\sigma,0}$.
However, as using this lattice facilitates some of the technical aspects,
others become more difficult. In particular, we strongly use the divisibilty
by powers of $p$ of certain binomial coefficients, which is not possible
when using this alternative lattice. Therefore, we have not been able
to use different lattices in order to prove more cases of the conjecture.
We have further hypothesized the possibility of using different lattices
for different values of $v_{F}(a)$, but this as well did not yield
any results.
\end{rem}
\begin{example}
We denote by $\underline{\rho}_{\underline{k}}^{0}$ the representation
of $KZ$ on the following space
\[
V_{\underline{\rho}_{\underline{k}}^{0}}=\bigotimes_{\sigma\in S}\underline{\rho}_{k_{\sigma}}^{\sigma,0}
\]

on which an element $\left(\begin{array}{cc}
a & b\\
c & d
\end{array}\right)\in KZ$ acts via
\begin{equation}
\underline{\rho}_{\underline{k}}^{0}\left(\begin{array}{cc}
a & b\\
c & d
\end{array}\right)\left(\bigotimes_{\sigma\in S}w_{\sigma}\right)=\bigotimes_{\sigma\in S}\left(\left(\begin{array}{cc}
a & b\\
c & d
\end{array}\right)w_{\sigma}\right)\label{eq:tensor product action}
\end{equation}

The example \ref{exa:separated lattice} assures us that the $\mathcal{O}_{C}$-module
$V_{\underline{\rho}_{\underline{k}}^{0}}$ is a separated lattice
of the space $V_{\underline{\rho}_{\underline{k}}}$ constructed in
Section \ref{subsec:-algebraic-representations-of G}. Therefore,
the $\mathcal{O}_{C}$-module $ind_{KZ}^{G}\underline{\rho}_{\underline{k}}^{0}$
is also a separated lattice of $ind_{KZ}^{G}\underline{\rho}_{\underline{k}}$
and is, by construction, stable under the action of $G$.

By Remark \ref{rem:induction as tensor product}, we can deduce the
existence of an injective map $\mathcal{H}(KZ,\underline{\rho}_{\underline{k}}^{0})\rightarrow\mathcal{H}(KZ,\underline{\rho}_{\underline{k}})$.
Moreover, one verifies that the operator $T\in\mathcal{H}(KZ,\underline{\rho}_{\underline{k}})$
defined in Section \ref{subsec:-algebraic-representations-of G},
induces by restriction a $G$-equivariant endomorphism of $ind_{KZ}^{G}\underline{\rho}_{\underline{k}}^{0}$
, which we again denote by $T$.

The following Lemma is proved in \cite[Lemma 3.3]{de2013existence},
but for sake of completeness we include here a proof of both isomorphisms.
\end{example}
\begin{lem}
\label{lem:Hecke algebra is polynomials}There are isomorphisms of
$\mathcal{O}_{C}$-algebras $\mathcal{H}_{\underline{\rho}_{\underline{k}}^{0}}(KZ,G)\simeq\mathcal{O}_{C}[T]$
and $\mathcal{H}_{\underline{\rho}_{\underline{k}}}(KZ,G)\simeq C[T]$.
\end{lem}
\begin{proof}
The space $V_{\underline{\rho}_{\underline{k}}}$ is an absolutely
irreducible $\mathfrak{g}\otimes_{\mathbb{Q}_{p}}C$-module, hence
by Lemma \ref{lem:isomorphism of Hecke algebras in the abs irred case},
$\iota_{\underline{\rho}_{\underline{k}}}$ is an isomorphism of $C$-algebras.
Lemma \ref{lem:hecke algebra is convolution} shows that there exists
a unique morphism of $C$-algebras $u_{\underline{\rho}_{\underline{k}}}:\mathcal{H}_{C}(KZ,G)\rightarrow\mathcal{H}_{\underline{\rho}_{\underline{k}}}(KZ,G)$
making the following diagram commute
\begin{equation}
\xymatrix{\mathcal{H}_{KZ}(C)\ar[r]^{\sim}\ar[d]^{\iota_{\underline{\rho}_{\underline{k}}}} & \mathcal{H}_{C}(KZ,G)\ar[d]^{u_{\underline{\rho}_{\underline{k}}}}\\
\mathcal{H}_{KZ}(\underline{\rho}_{\underline{k}})\ar[r]^{\sim} & \mathcal{H}_{\underline{\rho}_{\underline{k}}}(KZ,G)
}
\label{eq:hecke isomorphism}
\end{equation}
By construction, this morphism is an isomorphism of $C$-algebras.
Denote by $T_{1}\in\mathcal{H}_{C}(KZ,G)$ the element corresponding
to $\mathbf{1}_{KZ\alpha^{-1}KZ}\in\mathcal{H}_{KZ}(C)$ by Frobenius
reciprocity.

If $\varphi\in\mathcal{H}_{KZ}(C)$, then as it has compact support,
by the Cartan decomposition (\ref{eq:Cartan Decomposition}), it is
supported on $\coprod_{i=0}^{n}KZ\alpha^{-i}KZ$ for some integer
$n$. As $\varphi$ is $KZ$-bi-invariant (recall that $C$ is the
trivial representation), its restriction to each $KZ\alpha^{-i}KZ$
is constant, hence we may write $\varphi=\sum_{i=0}^{n}\varphi_{i}\cdot\mathbf{1}_{KZ\alpha^{-i}KZ}$.
Let $T_{i}\in\mathcal{H}_{C}(KZ,G)$ be the operator corresponding
to $\mathbf{1}_{KZ\alpha^{-i}KZ}$ by Frobenius reciprocity. Then
we see that the $T_{n}$'s span $\mathcal{H}_{C}(KZ,G)$ over $C$.
Geometrically, $T_{n}$ is the operator associating to a vertex $s$
the sum of the vertices at distance $n$ from $s$: this is because
\[
\mathbf{1}_{KZ\alpha^{-n}KZ}=\sum_{KZx\in KZ\backslash KZ\alpha^{-n}KZ}\mathbf{1}_{KZx}=
\]
\[
=\sum_{KZx\in KZ\backslash KZ\alpha^{-n}KZ}[x^{-1},1]=\sum_{KZx\in KZ\backslash KZ\alpha^{-n}KZ}x^{-1}\cdot[1,1]
\]
and then the $x^{-1}s_{0}$ are all distinct and give all vertices
$s'\in\mathcal{T}_{0}$ such that $s'$ is $KZ$-equivalent to $s_{n}=\alpha^{-n}s_{0}$.
This means that $(s_{0},s')$ is equivalent to $(s_{0},s_{n})$, which
is precisely our assertion. From the geometrical description of $T_{n}$,
one gets directly, since the tree $\mathcal{T}$ is $(q+1)$-regular,
that
\begin{eqnarray*}
T_{1}^{2} & = & T_{2}+(q+1)Id\\
T_{1}T_{n-1} & = & T_{n}+qT_{n-2}\quad\quad\forall n\ge3
\end{eqnarray*}

It follows that for all $n$, $T_{n}\in\mathcal{O}_{C}[T_{1}]$ is
monic of degree $n$. In particular, $\mathcal{H}_{C}(KZ,G)\simeq C[T_{1}]$.
Since $u_{\underline{\rho}_{\underline{k}}}(T_{1})=T$, it follows
that $\mathcal{H}_{\underline{\rho}_{\underline{k}}}(KZ,G)\simeq C[T]$.

Let us show that the restriction of this isomorphism to $\mathcal{H}_{\underline{\rho}_{\underline{k}}^{0}}(KZ,G)$
has image $\mathcal{O}_{C}[T]$.

As $T\in\mathcal{H}_{\underline{\rho}_{\underline{k}}^{0}}(KZ,G)$,
clearly $\mathcal{O}_{C}[T]$ is contained in the image. Let $p(T)\in C[T]$
be a polynomial corresponding to an element in $\mathcal{H}_{\underline{\rho}_{\underline{k}}^{0}}(KZ,G)$.

Assume $\deg(p)=n$, and let $a_{n}$ be the leading coefficient,
i.e. $p(T)=a_{n}T^{n}+p_{n-1}(T)$, where $\deg(p_{n-1})=n-1$. It
follows that $p(T)=a_{n}T_{n}+q_{n-1}(T)$, for some $q$ with $\deg(q_{n-1})=n-1$.

We recall that $T_{n}$ is the image under the natural isomorphisms
of $\mathbf{1}_{KZ\alpha^{-n}KZ}\in H_{KZ}(C)$, which maps to $\mathbf{1}_{KZ\alpha^{-n}KZ}\cdot\underline{\rho}_{\underline{k}}\in H_{KZ}(\underline{\rho}_{\underline{k}})$,
finally mapping to
\[
T_{n}([g,v])=\sum_{xKZ\in G/KZ}[gx,\mathbf{1}_{KZ\alpha^{-n}KZ}(x^{-1})\underline{\rho}_{\underline{k}}(x^{-1})(v)]=
\]
\[
=\sum_{xKZ\in KZ\alpha^{-n}KZ/KZ}[gx,\underline{\rho}_{\underline{k}}(x^{-1})(v)]
\]

Since $\alpha^{n}\in KZ\alpha^{-n}KZ$, and polynomials of order less
than $n$ are supported on $\coprod_{i=0}^{n-1}KZ\alpha^{-i}KZ$,
it follows that for any $v\in\underline{\rho}_{\underline{k}}$, one
has
\[
(p(T)([1,v]))(\alpha^{n})=\left(a_{n}T_{n}([1,v])\right)(\alpha^{n})=a_{n}\underline{\rho}_{\underline{k}}(\alpha^{-n})(v)=a_{n}U_{\underline{k}}^{n}(v)
\]

where the right most equality follows from (\ref{eq:U is rho}).

In particular, taking $v=\bigotimes_{\sigma:F\hookrightarrow C}y_{\sigma}^{k_{\sigma}}$,
we see that $v\in\underline{\rho}_{\underline{k}}^{0}$, hence $[1,v]\in ind_{KZ}^{G}\underline{\rho}_{\underline{k}}^{0}$.
As we assume $p(T)\in\mathcal{H}_{\underline{\rho}_{\underline{k}}^{0}}(KZ,G)=End_{\mathcal{O}_{C}[G]}(ind_{KZ}^{G}\underline{\rho}_{\underline{k}}^{0})$,
it follows that $p(T)([1,v])\in ind_{KZ}^{G}\underline{\rho}_{\underline{k}}^{0}$,
hence $a_{n}U_{\underline{k}}^{n}(v)=(p(T)([1,v]))(\alpha^{n})\in\underline{\rho}_{\underline{k}}^{0}$.
But, by definition of $U$, we see that $U_{\underline{k}}(v)=v$,
hence $a_{n}v\in\underline{\rho}_{\underline{k}}^{0}$.

However, by definition of $\underline{\rho}_{\underline{k}}^{0}$,
this is possible if and only if $a_{n}\in\mathcal{O}_{C}$. Therefore,
we see that $a_{n}T^{n}\in\mathcal{O}_{C}[T]$, and it suffices to
prove the claim for $p(T)-a_{n}T^{n}=p_{n-1}(T)$, which is a polynomial
of degree less than $n$.

Proceeding by induction, where the induction basis consists of constant
polynomials, which can be integral if and only if they belong to $\mathcal{O}_{C}$,
we conclude that $p(T)\in\mathcal{O}_{C}[T]$.
\end{proof}

\subsection{\label{subsec:Formulas}Formulas}

We keep the notations of Sections \ref{subsec:-algebraic-representations-of G}
and \ref{subsec:Lattices}. For $0\le m\le n$, we denote by $[\cdot]_{m}:I_{n}\rightarrow I_{m}$
the ``truncation'' map, defined by:
\[
\left[\sum_{i=0}^{n-1}\varpi^{i}[\mu_{i}]\right]_{m}=\begin{cases}
\sum_{i=0}^{m-1} & \varpi^{i}[\mu_{i}]\quad m\ge1\\
0 & m=0
\end{cases}
\]

For $\mu\in I_{n}$, we denote
\[
\lambda_{\mu}=\frac{\mu-[\mu]_{n-1}}{\varpi^{n-1}}\in I_{1}
\]

so that if $\mu=\sum_{i=0}^{n-1}\varpi^{i}[\mu_{i}]$, then $\lambda_{\mu}=[\mu_{n-1}]$.

We then have the following two results (see \cite{breuil2003quelques2,de2013existence}),
where $\psi$ denotes the function defined in Lemma \ref{lem:existence of the Hecke operator}.
\begin{lem}
\label{lem: T action on left side}Let $n\in\mathbb{N},\mu\in I_{n}$,
and let $v\in V_{\underline{\rho}_{\underline{k}}^{0}}$. We have:
\[
T\left([g_{n,\mu}^{0},v]\right)=T^{+}\left([g_{n,\mu}^{0},v]\right)+T^{-}\left([g_{n,\mu}^{0},v]\right)
\]

where
\[
T^{+}\left([g_{n,\mu}^{0},v]\right):=\sum_{\lambda\in I_{1}}\left[g_{n+1,\mu+\varpi^{n}\lambda,}^{0}\left(\underline{\rho}_{\underline{k}}(w)\circ\psi(\alpha^{-1})\circ\underline{\rho}_{\underline{k}}(w_{\lambda})\right)(v)\right]
\]
and
\[
T^{-}\left([g_{n,\mu}^{0},v]\right):=\begin{cases}
\left[g_{n-1,[\mu]_{n-1}}^{0},\left(\underline{\rho}_{\underline{k}}(ww_{-\lambda_{\mu}})\circ\psi(\alpha^{-1})\right)(v)\right] & n\ge1\\{}
[\alpha,\psi(\alpha^{-1})(v)] & n=0
\end{cases}
\]
\end{lem}
~
\begin{lem}
\label{lem:T action on right side}Let $n\in\mathbb{N},\mu\in I_{n}$,
and let $v\in V_{\underline{\rho}_{\underline{k}}^{0}}$. We have:
\[
T\left([g_{n,\mu}^{1},v]\right)=T^{+}\left([g_{n,\mu}^{1},v]\right)+T^{-}\left([g_{n,\mu}^{1},v]\right)
\]

where
\[
T^{+}\left([g_{n,\mu}^{1},v]\right):=\sum_{\lambda\in I_{1}}\left[g_{n+1,\mu+\varpi^{n}\lambda,}^{1}\left(\psi(\alpha^{-1})\circ\underline{\rho}_{\underline{k}}(w_{\lambda}w)\right)(v)\right]
\]
and
\[
T^{-}\left([g_{n,\mu}^{1},v]\right):=\begin{cases}
\left[g_{n-1,[\mu]_{n-1}}^{1},\left(\underline{\rho}_{\underline{k}}(w_{-\lambda_{\mu}})\circ\psi(\alpha^{-1})\circ\underline{\rho}_{\underline{k}}(w)\right)(v)\right] & n\ge1\\
\left[Id,\left(\underline{\rho}_{\underline{k}}(w)\circ\psi(\alpha^{-1})\circ\underline{\rho}_{\underline{k}}(w)\right)(v)\right] & n=0
\end{cases}
\]
\end{lem}
By using the equality $g_{n,\mu}^{1}=\beta g_{n,\mu}^{0}w$, these
two Lemmata yield the following two equalities:
\begin{eqnarray*}
T^{+}([g_{n,\mu}^{1},v]) & = & \beta T^{+}([g_{n,\mu}^{0},\underline{\rho}_{\underline{k}}(w)(v)])\\
T^{-}([g_{n,\mu}^{1},v]) & = & \beta T^{-}([g_{n,\mu}^{0},\underline{\rho}_{\underline{k}}(w)(v)])
\end{eqnarray*}
and also the following result
\begin{cor}
Let $n\in\mathbb{N}$, $\mu,\lambda\in I_{n}$, $i,j\in\{0,1\}$ and
$v_{1},v_{2}\in V_{\underline{\rho}_{\underline{k}}^{0}}$. If $i\ne j$
or if $\mu\ne\lambda$, then $T^{+}([g_{n,\mu}^{i},v_{1}])$ and $T^{+}([g_{n,\lambda}^{j},v_{2}])$
have disjoint supports.
\end{cor}
The following Lemma is a simple generalization of \cite{breuil2003quelques2},
Lemma 2.2.2.
\begin{lem}
Let $v=\sum_{\underline{0}\le\underline{i}\le\underline{k}}c_{\underline{i}}e_{\underline{k},\underline{i}}\in V_{\rho_{\underline{k}}^{0}}$
and $\lambda\in\mathcal{O}_{F}$. We have:
\begin{equation}
\left(\rho_{\underline{k}}(w)\circ\psi(\alpha^{-1})\circ\rho_{\underline{k}}(w_{\lambda})\right)(v)=\sum_{\underline{0}\le\underline{j}\le\underline{k}}\left(\varpi^{\underline{j}}\sum_{\underline{j}\le\underline{i}\le\underline{k}}c_{\underline{i}}{\underline{i} \choose \underline{j}}(-\lambda)^{\underline{i}-\underline{j}}\right)e_{\underline{k},\underline{j}}\label{eq:T+ formula}
\end{equation}

\begin{equation}
\left(\rho_{\underline{k}}(ww_{\lambda})\circ\psi(\alpha^{-1})\right)(v)=\sum_{\underline{0}\le\underline{j}\le\underline{k}}\left(\sum_{\underline{j}\le\underline{i}\le\underline{k}}\varpi^{\underline{k}-\underline{i}}{\underline{i} \choose \underline{j}}c_{\underline{i}}(-\lambda)^{\underline{i}-\underline{j}}\right)e_{\underline{k},\underline{j}}\label{eq:T- formula}
\end{equation}
\begin{equation}
\psi(\alpha^{-1})(v)=\sum_{\underline{0}\le\underline{j}\le\underline{k}}\varpi^{\underline{k}-\underline{j}}c_{\underline{j}}e_{\underline{k},\underline{j}}\label{eq:psi formula}
\end{equation}
\end{lem}
\begin{proof}
Equation (\ref{eq:T+ formula}) is proved in \cite{de2013existence}
and equation (\ref{eq:psi formula}) is immediate. For equation (\ref{eq:T- formula}),
we note that by equation (\ref{eq:formula for G algebraic action}),
we have for any $\sigma\in S$ and any $0\le i_{\sigma}\le k_{\sigma}$:
\[
\left(w\circ w_{\lambda}\circ U_{k_{\sigma}}\right)(e_{k_{\sigma},i_{\sigma}})=\left(\begin{array}{cc}
1 & -\lambda\\
0 & 1
\end{array}\right)\left(\sigma(\varpi)^{k_{\sigma}-i_{\sigma}}e_{k_{\sigma},i_{\sigma}}\right)=
\]
\[
=\sigma(\varpi)^{k_{\sigma}-i_{\sigma}}\cdot x^{k_{\sigma}-i_{\sigma}}(y+\sigma(-\lambda)x)^{i_{\sigma}}=
\]
\[
=\sigma(\varpi)^{k_{\sigma}-i_{\sigma}}\cdot\sum_{j_{\sigma}=0}^{i_{\sigma}}{i_{\sigma} \choose j_{\sigma}}\sigma(-\lambda)^{i_{\sigma}-j_{\sigma}}x^{k_{\sigma}-j_{\sigma}}y^{j_{\sigma}}=
\]
\[
=\sigma(\varpi)^{k_{\sigma}-i_{\sigma}}\cdot\sum_{j_{\sigma}=0}^{i_{\sigma}}{i_{\sigma} \choose j_{\sigma}}\sigma(-\lambda)^{i_{\sigma}-j_{\sigma}}e_{k_{\sigma},j_{\sigma}}
\]

Using equation (\ref{eq:tensor product action}), we deduce that
\[
\left(\rho_{\underline{k}}(ww_{\lambda})\circ\psi(\alpha^{-1})\right)(v)=\sum_{\underline{0}\le i\le\underline{k}}c_{\underline{i}}\cdot\bigotimes_{\sigma\in S}\left(w\circ w_{\lambda}\circ U_{k_{\sigma}}\right)(e_{k_{\sigma},i_{\sigma}})=
\]
\[
=\sum_{\underline{0}\le\underline{i}\le\underline{k}}c_{\underline{i}}\cdot\bigotimes_{\sigma\in S}\left(\sigma(\varpi)^{k_{\sigma}-i_{\sigma}}\cdot\sum_{j_{\sigma}=0}^{i_{\sigma}}{i_{\sigma} \choose j_{\sigma}}\sigma(-\lambda)^{i_{\sigma}-j_{\sigma}}e_{k_{\sigma},j_{\sigma}}\right)=
\]
\[
=\sum_{\underline{0}\le\underline{i}\le\underline{k}}c_{\underline{i}}\cdot\prod_{\sigma\in S}\sigma(\varpi)^{k_{\sigma}-i_{\sigma}}\cdot\sum_{\underline{0}\le\underline{j}\le\underline{i}}\prod_{\sigma\in S}{i_{\sigma} \choose j_{\sigma}}\cdot\prod_{\sigma\in S}\sigma(-\lambda)^{i_{\sigma}-j_{\sigma}}e_{\underline{k},\underline{j}}=
\]
\[
=\sum_{\underline{0}\le\underline{j}\le\underline{k}}\left(\sum_{\underline{j}\le\underline{i}\le\underline{k}}\varpi^{\underline{k}-\underline{i}}\cdot{\underline{i} \choose \underline{j}}\cdot c_{\underline{i}}\cdot(-\lambda)^{\underline{i}-\underline{j}}\right)e_{\underline{k},\underline{j}}
\]
\end{proof}
This leads to the following corollary, which is a simple generalization
of \cite{breuil2003quelques2}, Corollary 2.2.3.
\begin{cor}
\label{cor:definition of coefficients}Let $m\in\mathbb{Z}_{>0}$,
$a\in C$, and for any $\mu\in I_{m}$ (resp. $\mu\in I_{m-1}$, resp.
$\mu\in I_{m+1}$), $v_{\mu}^{m}=\sum_{\underline{0}\le\underline{i}\le\underline{k}}c_{\underline{i},\mu}^{m}\cdot e_{\underline{k},\underline{i}}$
(resp. $v_{\mu}^{m-1}=\sum_{\underline{0}\le\underline{i}\le\underline{k}}c_{\underline{i},\mu}^{m-1}\cdot e_{\underline{k},\underline{i}}$,
resp. $v_{\mu}^{m+1}=\sum_{\underline{0}\le\underline{i}\le\underline{k}}c_{\underline{i},\mu}^{m+1}\cdot e_{\underline{k},\underline{i}}$)
an element of $\underline{\rho}_{\underline{k}}$. We denote
\[
f_{m}=\sum_{\mu\in I_{m}}[g_{m,\mu}^{0},v_{\mu}^{m}]
\]
\[
f_{m-1}=\sum_{\mu\in I_{m-1}}[g_{m-1,\mu}^{0},v_{\mu}^{m-1}]
\]
\[
f_{m+1}=\sum_{\mu\in I_{m+1}}[g_{m+1,\mu}^{0},v_{\mu}^{m+1}]
\]

Then
\[
T^{-}(f_{m+1})+T^{+}(f_{m-1})-af_{m}=\sum_{\mu\in I_{m}}\left[g_{m,\mu}^{0},\sum_{\underline{0}\le\underline{j}\le\underline{k}}C_{\underline{j},\mu}^{m}\cdot e_{\underline{k},\underline{j}}\right]
\]
where
\begin{equation}
\begin{split}
C_{\underline{j},\mu}^{m}=\sum_{\underline{j}\le\underline{i}\le\underline{k}}\varpi^{\underline{k}-\underline{i}}{\underline{i} \choose \underline{j}}\cdot\sum_{\lambda\in k_{F}}c_{\underline{i},\mu+\varpi^{m}[\lambda]}^{m+1}\cdot[\lambda]^{\underline{i}-\underline{j}}+\\
+\varpi^{\underline{j}}\cdot\sum_{\underline{j}\le\underline{i}\le\underline{k}}c_{\underline{i},[\mu]_{m-1}}^{m-1}{\underline{i} \choose \underline{j}}\left(-\lambda_{\mu}\right)^{\underline{i}-\underline{j}}-ac_{\underline{j},\mu}^{m}\label{eq:formula for the coefficients}
\end{split}
\end{equation}
\end{cor}

\section{A Criterion for Separability}

\subsection{The main result}

We adhere to the notations of Sections \ref{subsec:Lattices} and
\ref{subsec:Formulas} and fix an embedding $\iota:F\hookrightarrow C$.
Denote
\[
S^{+}=\{\sigma\in S\mid k_{\sigma}\ne0\}\subseteq S
\]

We partition $S^{+}$ with respect to the action of $\sigma\in S^{+}$
on the residue field of $F$. More precisely, for any $l\in\{0,\ldots,f-1\}$,
we let
\[
J_{l}=\left\{ \sigma\in S^{+}\mid\sigma(\lambda)=\iota\circ\varphi_{0}^{l}(\lambda)\quad\forall\lambda\in I_{1}\right\}
\]
where $I_{1}=\left\{ [\zeta]\mid\zeta\in\kappa_{F}\right\} $. In
particular, we remark that
\[
\coprod_{l=0}^{f-1}J_{l}=S^{+},\quad\forall l\in\{0,\ldots,f-1\}\quad\left|J_{l}\right|\le e.
\]

For any integer $i\in\mathbb{Z}$, we denote by $\overline{i}$ the
unique representative of $i\mod f$ in $\{0,\ldots,f-1\}$. We also
let, for any $\sigma\in J_{l}$, $\gamma_{\sigma}:=l$ and
\[
v_{\sigma}=\inf\left\{ i\mid1\le i\le f,J_{\overline{l+i}}\ne\emptyset\right\}
\]
that is the minimal power of Frobenius $\varphi_{0}$ needed to pass
from $J_{l}$ to another nonempty $J_{k}$.

Let $a\in\mathfrak{p}_{C}$. We let
\[
\Pi_{\underline{k},a}=\frac{ind_{KZ}^{G}\underline{\rho}_{\underline{k}}}{(T-a)(ind_{KZ}^{G}\underline{\rho}_{\underline{k}})}.
\]
This is a locally algebraic representation of $G$, which can be realized
as the tensor product of an algebraic representation with a smooth
representation. More precisely, we have the following result, which
is stated in \cite{de2013existence}.
\begin{prop}
Let $u_{\sigma}=\frac{k_{\sigma}}{2}$ for any $\sigma\in S$.

$(i)$ If $a\notin\{\pm((q+1)\varpi^{\underline{u}}\}$, then $\Pi_{\underline{k},a}$
is algebraicly irreducible and
\[
\Pi_{\underline{k},a}\simeq\underline{\rho}_{\underline{k}}\otimes Ind_{B}^{G}(\mbox{nr}(\lambda_{1}^{-1})\otimes\mbox{nr}(\lambda_{2}^{-1}))
\]
where $\lambda_{1},\lambda_{2}$ satisfy
\[
\lambda_{1}\lambda_{2}=\varpi^{\underline{k}},\quad\lambda_{1}+q\lambda_{2}=a
\]

$(ii)$ If $a\in\{\pm((q+1)\varpi^{\underline{u}}\}$, then we have
a short exact sequence
\[
0\rightarrow\underline{\rho}_{\underline{k}}\otimes\mbox{St}_{G}\otimes(\mbox{nr}(\delta)\circ\det)\rightarrow\Pi_{\underline{k},a}\rightarrow\underline{\rho}_{\underline{k}}\otimes(\mbox{nr}(\delta)\circ\det)\rightarrow0
\]
where $\mbox{St}_{G}=C^{0}(\mathbf{P}^{1}(F),C)/\mbox{\{constants\}}$
denotes the Steinberg representation of $G$ and where $\delta=(q+1)/a$.
\end{prop}
As in \cite{de2013existence}, we define
\[
\Theta_{\underline{k},a}=\mbox{Im}\left(ind_{KZ}^{G}\underline{\rho}_{\underline{k}}^{0}\rightarrow\Pi_{\underline{k},a}\right)
\]
which is the same as
\[
\Theta_{\underline{k},a}=\frac{ind_{KZ}^{G}\underline{\rho}_{\underline{k}}^{0}}{ind_{KZ}^{G}\underline{\rho}_{\underline{k}}^{0}\cap(T-a)(ind_{KZ}^{G}\underline{\rho}_{\underline{k}})}.
\]
This is a lattice in $\Pi_{\underline{k},a}$ and, since $ind_{KZ}^{G}\rho_{\underline{k}}^{0}$
is a finitely generated $\mathcal{O}_{C}[G]$-module, we see that
$\Theta_{\underline{k},a}$ is also a finitely generated $\mathcal{O}_{C}[G]$-module.

Now, the Breuil Schneider conjecture \ref{conj:The-following-two}
asserts that $\underline{\rho}_{\underline{k}}\otimes\pi$ admits
a $G$-invariant norm.

By \cite[Prop. 1.17]{emerton2005p}, this is equivalent to the existence
of a separated lattice, and even to any finitely generated lattice
being separated.

The following conjecture is then a restatement of the Breuil-Schneider
conjecture.
\begin{conjecture}
The $\mathcal{O}_{C}$-module $\Theta_{\underline{k},a}$ does not
contain any $C$-line (it is separated).
\end{conjecture}
We also recall that Breuil, in \cite{breuil2003quelques2} proves
the conjecture for $F=\mathbb{Q}_{p}$ and $k<2p-1$, and that de
Ieso, in \cite{de2013existence}, proves it when $|J_{l}|\le1$ for
all $l\in\{0,\ldots,f-1\}$ and for any $\sigma\in S^{+}$, $k_{\sigma}+1\le p^{v_{\sigma}}$.

The idea, as in \cite{breuil2003quelques2}, is to reduce the problem
to a statement which we can prove inductively, sphere by sphere.

As we shall use that idea repeatedly, we introduce a related definition.
Abusing notation, we denote by $B_{N}\subseteq ind_{KZ}^{G}\underline{\rho}_{\underline{k}}$
the set of functions supported in $B_{N}=B_{N}^{0}\coprod B_{N-1}^{1}$,
where $B_{N}^{0}=\coprod_{M\le N}S_{M}^{0}$, $B_{N}^{1}=\coprod_{M\le N}S_{M}^{1}$,
and we have defined
\[
S_{M}^{0}=I\alpha^{-M}KZ,\quad S_{M}^{1}=I\beta\alpha^{-M}KZ
\]
We also recall that $B^{0},B^{1}$ denote the sets of functions supported on 

$\bigcup_{N}B_{N}^{0},\bigcup_{N}B_{N}^{1}$, respectively.
\begin{defn}
Let $\underline{k}\in\mathbb{N}^{S}$, and let $a\in\mathcal{O}_{C}$.
We say that the pair $(\underline{k},a)$ is \emph{separated }if for
all $N\in\mathbb{Z}_{>0}$ large enough, there exists a constant $\epsilon\in\mathbb{Z}_{\ge0}$
depending only on $N,\underline{k},a$ such that for all $n\in\mathbb{Z}_{\ge0}$,
and all $f\in B^{0}$
\begin{equation}
(T-a)(f)\in B_{N}+\varpi^{n}ind_{KZ}^{G}\underline{\rho}_{\underline{k}}^{0}\Rightarrow f\in B_{N-1}+\varpi^{n-\epsilon}ind_{KZ}^{G}\underline{\rho}_{\underline{k}}^{0}\label{eq:Main Theorem implication}
\end{equation}
\end{defn}
\begin{rem}
We slightly abuse notation here, as $\varpi\notin C$, but as $v_{F}(\sigma(\varpi))=v_{F}(\varpi)=1$
for all $\sigma\in S$, one may choose any embedding $\sigma:F\hookrightarrow C$,
and consider $\sigma(\varpi)^{n}$ instead.
\end{rem}
The upshot is that we have the following result.
\begin{thm}
Let $\underline{k}\in\mathbb{N}^{S}$, let $a\in\mathcal{O}_{C}$.
If $(\underline{k},a)$ is separated, then $\Theta_{\underline{k},a}$
is separated.
\end{thm}
\begin{proof}
First, note that if (\ref{eq:Main Theorem implication}) holds for
all $f\in ind_{KZ}^{G}\underline{\rho}_{\underline{k}}$, then the
proof of \cite{breuil2003quelques2}, Corollary 4.1.2 shows that
$\Theta_{\underline{k},a}$ is separated.

Next, for an arbitary $f\in ind_{KZ}^{G}\underline{\rho}_{\underline{k}}$,
write $f=f^{0}+f^{1}$ with $f^{0}\in B^{0}$ and $f^{1}\in B^{1}$.
Then by the formulas in Lemma \ref{lem: T action on left side} and
Lemma \ref{lem:T action on right side}, it follows that
\[
\mbox{supp}\left((T-a)(f^{0})\right)\cap\mbox{supp}\left((T-a)(f^{1})\right)\subseteq S_{0}=B_{0}\subseteq B_{N}
\]
If we assume that
\[
(T-a)(f^{0})+(T-a)(f^{1})=(T-a)(f)\in B_{N}+\varpi^{n}ind_{KZ}^{G}\underline{\rho}_{\underline{k}}^{0}
\]
it follows that both $(T-a)(f^{0})\in B_{N}+\varpi^{n}ind_{KZ}^{G}\underline{\rho}_{\underline{k}}^{0}$
and $(T-a)(f^{1})\in B_{N}+\varpi^{n}ind_{KZ}^{G}\underline{\rho}_{\underline{k}}^{0}$.

Since $f^{0}\in B^{0}$ and $(\underline{k},a)$ is separated, it
follows that $f^{0}\in B_{N-1}+\varpi^{n-\epsilon}ind_{KZ}^{G}\underline{\rho}_{\underline{k}}^{0}$.

Moreover, since $T$ is $G$-equivariant, and $\varpi^{\mathbb{Z}}\cdot Id$
acts trivially, we see that
\[
\beta(T-a)(\beta f^{1})=(T-a)(f^{1})\in B_{N}+\varpi^{n}ind_{KZ}^{G}\underline{\rho}_{\underline{k}}^{0}
\]
Since $\beta$ acts by translation, it does not affect the values
of the function, and since $\beta B_{N}=B_{N}$, it follows that
\[
(T-a)(\beta f^{1})\in B_{N}+\varpi^{n}ind_{KZ}^{G}\underline{\rho}_{\underline{k}}^{0}
\]
with $\beta f^{1}\in B^{0}$. Since $(\underline{k},a)$ is separated,
we get $\beta f^{1}\in B_{N-1}+\varpi^{n-\epsilon}ind_{KZ}^{G}\underline{\rho}_{\underline{k}}^{0}$,
hence $f^{1}\in B_{N-1}+\varpi^{n-\epsilon}ind_{KZ}^{G}\underline{\rho}_{\underline{k}}^{0}$.

In conclusion
\[
f=f^{0}+f^{1}\in B_{N-1}+\varpi^{n-\epsilon}ind_{KZ}^{G}\underline{\rho}_{\underline{k}}^{0}
\]
as claimed.
\end{proof}
It therefore remains to show that certain pairs $(\underline{k},a)$
are separated.

In this section, we will prove the following theorem:
\begin{thm}
\label{thm:Main Theorem}Assume that $\left|S^{+}\right|=1$, denote
by $\sigma$ the unique element in $S^{+}$, and and let $k=k_{\sigma}=d\cdot q+r$,
with $0\le r<q$. Assume that one of the following three conditions
is satisfied:

$(i)$ $k\le\frac{1}{2}q^{2}$ with $r<q-d$ and $v_{F}(a)\in[0,1]$.

$(ii)$ $k\le\frac{1}{2}q^{2}$ with $2v_{F}(a)-1\le r<q-d$ and $v_{F}(a)\in[1,e]$.

$(iii)$ $k\le\min\left(p\cdot q-1,\frac{1}{2}q^{2}\right)$ , $d-1\le r$
and $v_{F}(a)\ge d$.

Then $(\underline{k},a)$ is separated.
\end{thm}
\begin{cor}
Under the above conditions, $\Theta_{\underline{k},a}$ is separated,
hence $\Pi_{\underline{k},a}$ admits an invariant norm.
\end{cor}
Since our assumptions include the fact that $\left|S^{+}\right|=1$,
we may proceed with the following notational simplifications.

We assume that $C$ contains $F$, and let $\sigma=\iota:F\hookrightarrow C$
be the natural inclusion.

We may further let $k=k_{\sigma}$ stand for the multi-index $\underline{k}$
corresponding to $k$, and similarly for all multi-indices.

\subsection{Preparation}

Before we prove the theorems, let us first prove the following useful
lemmata, which we will employ later on.
\begin{lem}
\label{subsec:Lemma divisibility by x^q-x} Let $\kappa$ be a finite
field of characteristic $p$ containing $\mathbb{F}_{q}$. Consider
a polynomial $h\in\kappa[x]$, such that
\[
h(x+\lambda)\in x^{j}\cdot\kappa[x]\quad\forall\lambda\in\mathbb{F}_{q}
\]

Then
\[
h(x)\in(x^{q}-x)^{j}\cdot\kappa[x]
\]
\end{lem}
\begin{proof}
We will prove the Lemma by induction on $j$. For $j=1$, $h(x+\lambda)\in x\cdot\kappa[x]$
implies that $h(\lambda)=0$ for all $\lambda\in\mathbb{F}_{q}$,
hence $x^{q}-x\mid h(x)$, as claimed.

In general, $h(x+\lambda)\in x^{j}\cdot\kappa[x]\subseteq x\cdot\kappa[x]$
for all $\lambda\in\mathbb{F}_{q}$, hence $h(x)=(x^{q}-x)\cdot g(x)$
for some $g(x)\in\kappa[x]$, by the $j=1$ case. But $\gcd(x^{q}-x,x^{j})=x$,
hence we get
\[
h(x+\lambda)=(x^{q}-x)\cdot g(x+\lambda)\in x^{j}\cdot\kappa[x]\Rightarrow g(x+\lambda)\in x^{j-1}\cdot\kappa[x]
\]

for all $\lambda\in\mathbb{F}_{q}$.

By the induction hypothesis, it follows that $g(x)\in(x^{q}-x)^{j-1}\cdot\kappa[x]$,
hence $h(x)\in(x^{q}-x)^{j}\cdot\kappa[x]$.
\end{proof}
\begin{lem}
\label{lem:x^q-x modulo pi}Let $k,d\in\mathbb{N}$. Let $h(x)=\sum_{i=0}^{k}c_{i}x^{i}\in\mathcal{O}_{C}[x]$
be such that for all $0\le j\le d$, and all $\lambda\in\mathbb{F}_{q}$,
we have
\[
\sum_{i=j}^{k}{i \choose j}c_{i}[\lambda]^{i-j}\in\varpi_{C}\mathcal{O}_{C}
\]
where $[\lambda]\in\mathcal{O}_{F}\hookrightarrow\mathcal{O}_{C}$
is the Teichmüller representative of $\lambda$. Then $h(x)\in(x^{q}-x)^{d+1}\cdot\mathcal{O}_{C}[x]+\varpi_{C}\cdot\mathcal{O}_{C}[x]$.
\end{lem}
\begin{proof}
By our assumption, since
\[
h(x+[\lambda])=\sum_{i=0}^{k}c_{i}(x+[\lambda])^{i}=\sum_{i=0}^{k}c_{i}\sum_{j=0}^{i}{i \choose j}x^{j}[\lambda]^{i-j}=
\]
\begin{equation}
=\sum_{j=0}^{k}\left(\sum_{i=j}^{k}{i \choose j}c_{i}[\lambda]^{i-j}\right)x^{j}\label{eq:f(x+lambda)}
\end{equation}
we see that
\[
h(x+[\lambda])\in\left(x^{d+1},\varpi_{C}\right)\quad\forall\lambda\in\mathbb{F}_{q}
\]
Equivalently, considering the image in $k_{C}=\mathcal{O}_{C}/\varpi_{C}\mathcal{O}_{C}$,
we have $\overline{h}\in\kappa_{C}[x]$ of degree at most $k$, satisfying

$\overline{h}(x+\lambda)\in(x^{d+1})$ for all $\lambda\in\mathbb{F}_{q}$.

By Lemma \ref{subsec:Lemma divisibility by x^q-x}, we see that $\overline{h}(x)\in(x^{q}-x)^{d+1}\cdot\kappa_{C}[x]$,
hence 
$$h(x)\in\left((x^{q}-x)^{d+1},\varpi_{C}\right).$$
This establishes the Lemma.
\end{proof}
\begin{lem}
\label{lem:x^q-x modulo pi^n}Let $n,k,d\in\mathbb{N}$. Let $f(x)=\sum_{i=0}^{k}c_{i}x^{i}\in\mathcal{O}_{C}[x]$
be such that for all $0\le j\le d$ and all $\lambda\in\mathbb{F}_{q}$
we have
\[
\sum_{i=j}^{k}{i \choose j}c_{i}[\lambda]^{i-j}\in\varpi_{C}^{n}\mathcal{O}_{C}
\]
where $[\lambda]\in\mathcal{O}_{F}\hookrightarrow\mathcal{O}_{C}$
is the Teichmüller representative of $\lambda$. Then $f(x)\in(x^{q}-x)^{d+1}\cdot\mathcal{O}_{C}[x]+\varpi_{C}^{n}\cdot\mathcal{O}_{C}[x]$.
\end{lem}
\begin{proof}
By induction on $n$. For $n=1$, this is Lemma \ref{lem:x^q-x modulo pi}.
Assume it holds for $n-1$, and let us prove it for $n$.

Since $\varpi_{C}^{n}\mathcal{O}_{C}\subseteq\varpi_{C}^{n-1}\mathcal{O}_{C}$,
the induction hypothesis implies that $f(x)\in\left((x^{q}-x)^{d+1},\varpi_{C}^{n-1}\right)$,
so we may write
\[
f(x)=(x^{q}-x)^{d+1}\cdot g(x)+\varpi_{C}^{n-1}\cdot h(x)
\]

By (\ref{eq:f(x+lambda)}), our assumption implies that
\[
f(x+[\lambda])\in\left(x^{d+1},\varpi_{C}^{n}\right)\quad\forall\lambda\in\mathbb{F}_{q}
\]
substituting in the above equation, we get
\[
\left((x+[\lambda])^{q}-(x+[\lambda])\right)^{d+1}\cdot g(x+[\lambda])+\varpi_{C}^{n-1}\cdot h(x+[\lambda])\in(x^{d+1},\varpi_{C}^{n})
\]
But
\[
(x+[\lambda])^{q}-(x+[\lambda])=\sum_{i=0}^{q}{q \choose i}[\lambda]^{q-i}x^{i}-x-[\lambda]=\sum_{i=1}^{q}{q \choose i}[\lambda]^{q-i}x^{i}-x\in x\cdot\mathcal{O}_{C}[x]
\]
since $[\lambda]^{q}=[\lambda]$ for all $\lambda\in\mathbb{F}_{q}$.
This shows that $\left((x+[\lambda])^{q}-(x+[\lambda])\right)^{d+1}\in(x^{d+1})\subseteq(x^{d+1},\varpi_{C}^{n})$,
hence
\[
\varpi_{C}^{n-1}\cdot h(x+[\lambda])\in(x^{d+1},\varpi_{C}^{n})\quad\forall\lambda\in\mathbb{F}_{q}
\]

which implies that
\[
h(x+[\lambda])\in(x^{d+1},\varpi_{C})\quad\forall\lambda\in\mathbb{F}_{q}
\]
Considering the reduction modulo $\varpi_{C}$, by Lemma \ref{subsec:Lemma divisibility by x^q-x},
it follows that $h(x)\in\left((x^{q}-x)^{d+1},\varpi_{C}\right)$,
hence
\[
f(x)\in\left((x^{q}-x)^{d+1}\right)+\varpi_{C}^{n-1}\cdot\left((x^{q}-x)^{d+1},\varpi_{C}\right)=\left((x^{q}-x)^{d+1},\varpi_{C}^{n}\right)
\]
establishing the claim.
\end{proof}
\begin{lem}
\label{lem:divisibility by x^q-x modulo pi^n}Let $n\in\mathbb{Z}$,
$k,d\in\mathbb{N}$. Let $f(x)=\sum_{i=0}^{k}c_{i}x^{i}\in C[x]$
be such that for all $0\le j\le d$ and all $\lambda\in\mathbb{F}_{q}$
we have
\[
\sum_{i=j}^{k}{i \choose j}c_{i}[\lambda]^{i-j}\in\varpi_{C}^{n}\mathcal{O}_{C}
\]
where $[\lambda]\in\mathcal{O}_{F}\hookrightarrow\mathcal{O}_{C}$
is the Teichmüller representative of $\lambda$. Then $f(x)\in(x^{q}-x)^{d+1}\cdot C[x]+\varpi_{C}^{n}\cdot\mathcal{O}_{C}[x]$.
\end{lem}
\begin{proof}
Let $L=\min_{0\le i\le k}v_{C}(c_{i})$. Consider $g(x)=\varpi_{C}^{-L}\cdot f(x)\in\mathcal{O}_{C}[x]$.
If $n\le L$, then as $f(x)\in\varpi_{C}^{L}\mathcal{O}_{C}[x]\subseteq\varpi_{C}^{n}\mathcal{O}_{C}[x]$,
we are done.

Else, $g(x)$ satisfies for all $0\le j\le d$ and all $\lambda\in\mathbb{F}_{q}$

\[
\sum_{i=j}^{k}{i \choose j}\varpi_{C}^{-L}\cdot c_{i}[\lambda]^{i-j}\in\varpi_{C}^{n-L}\mathcal{O}_{C}
\]
with $n-L\ge1$, hence by Lemma \ref{lem:x^q-x modulo pi^n}, $g(x)\in(x^{q}-x)^{d+1}\cdot\mathcal{O}_{C}[x]+\varpi_{C}^{n-L}\cdot\mathcal{O}_{C}[x]$,
hence $f(x)\in(x^{q}-x)^{d+1}\cdot C[x]+\varpi_{C}^{n}\cdot\mathcal{O}_{C}[x]$
.
\end{proof}
\begin{lem}
\label{subsec:Lemma divisibility of polynomial coefficients}Let $n\in\mathbb{Z}$
and let $k\in\mathbb{N}$. Let $d=\left\lfloor k/q\right\rfloor $.
Let $(c_{i})_{i=0}^{k}$ be a sequence in $C$ such that for all $0\le j\le d$,
and all $\lambda\in\mathbb{F}_{q}$, we have
\[
\sum_{i=j}^{k}{i \choose j}c_{i}[\lambda]^{i-j}\in\pi_{C}^{n}\mathcal{O}_{C}
\]
 where $[\lambda]\in\mathcal{O}_{F}\hookrightarrow\mathcal{O}_{C}$
is the Teichmüller representative of $\lambda$. Then $c_{i}\in\varpi_{C}^{n}\mathcal{O}_{C}$
for all $0\le i\le k$.
\end{lem}
\begin{proof}
By Lemma \ref{lem:divisibility by x^q-x modulo pi^n}, we see that
$f(x)=\sum_{i=0}^{k}c_{i}x^{i}\in(x^{q}-x)^{d+1}\cdot C[x]+\varpi_{C}^{n}\mathcal{O}_{C}[x]$,
but $\deg(f)\le k<q(d+1)$, hence $f(x)\in\varpi_{C}^{n}\mathcal{O}_{C}[x]$.
This establishes the Lemma.
\end{proof}
\begin{lem}
\label{subsec:Lemma sum coeffs divisibility} Let $k,d\in\mathbb{N}$.
Let $f(x)=\sum_{i=0}^{k}c_{i}x^{i}\in C[x]$ and let $n\in\mathbb{Z}$.
Assume that for all $0\le j\le d$, and all $\lambda\in\mathbb{F}_{q}$,
we have
\[
\sum_{i=j}^{k}{i \choose j}c_{i}[\lambda]^{i-j}\in\varpi_{C}^{n}\mathcal{O}_{C}
\]
 where $[\lambda]\in\mathcal{O}_{F}\hookrightarrow\mathcal{O}_{C}$
is the Teichmüller representative of $\lambda$. Then
\[
c_{i}\in\varpi_{C}^{n}\mathcal{O}_{C}\quad\forall0\le i\le d
\]
\begin{equation}
\sum_{l=d}^{\left\lfloor \frac{k-j}{q-1}\right\rfloor }{l \choose d}\cdot c_{j+l(q-1)}\in\varpi_{C}^{n}\mathcal{O}_{C}\quad\forall d+1\le j\le d+q-1\label{eq:coefficients relation modulo (x^q-x)^(d+1)}
\end{equation}

\end{lem}
\begin{proof}
By Lemma \ref{lem:divisibility by x^q-x modulo pi^n}, we see that
$f(x)\in(x^{q}-x)^{d+1}\cdot C[x]+\varpi_{C}^{n}\mathcal{O}_{C}[x]$.
We proceed by reducing $f(x)$ modulo $(x^{q}-x)^{d+1}$.

In order to do so, we first have to understand the reduction of a
general monomial $x^{t}$ modulo $(x^{q}-x)^{d+1}$.

We prove, by induction on $s$, that for every $0\le s\le\left\lfloor \frac{t-d-1}{q-1}\right\rfloor -d-1$
and every $t\ge q(d+1)$ we have
\begin{equation}
x^{t}\equiv\sum_{l=1}^{d+1}(-1)^{l+1}{d+1+s \choose l+s}\cdot{l+s-1 \choose s}x^{t-(l+s)(q-1)}\mod(x^{q}-x)^{d+1}\label{eq:x^t mod (x^q-x)^(d+1)}
\end{equation}

Indeed, for $s=0$, this is simply a restatement of the binomial expansion,
as
\[
x^{t}=x^{t-(d+1)\cdot q}\cdot x^{(d+1)\cdot q}\equiv x^{t-(d+1)\cdot q}\cdot\left(x^{(d+1)q}-(x^{q}-x)^{d+1}\right)=
\]
\[
=x^{t-(d+1)\cdot q}\cdot\left(x^{(d+1)q}-\sum_{l=0}^{d+1}(-1)^{l}{d+1 \choose l}\cdot\left(x^{q}\right)^{(d+1)-l}\cdot x^{l}\right)=
\]
\[
=x^{t-(d+1)\cdot q}\cdot\sum_{l=1}^{d+1}(-1)^{l+1}{d+1 \choose l}x^{(d+1)\cdot q-l(q-1)}=
\]
\[
=\sum_{l=1}^{d+1}(-1)^{l+1}{d+1 \choose l}x^{t-l(q-1)}\mod(x^{q}-x)^{d+1}
\]
Assume it holds for $s-1$, and let us prove it holds for $s$.

By the induction hypothesis
\begin{equation}
x^{t}\equiv\sum_{l=1}^{d+1}(-1)^{l+1}{d+s \choose l+s-1}{l+s-2 \choose s-1}x^{t-(l+s-1)(q-1)}\mod(x^{q}-x)^{d+1}\label{eq:induction step x^t mod (x^q-x)^(d+1)}
\end{equation}

Since $s\le\left\lfloor \frac{t-d-1}{q-1}\right\rfloor -d-1$, we
see that
\[
(q-1)(d+1+s)\le t-(d+1)\Rightarrow t-s(q-1)\ge q(d+1)
\]

This implies, by the case $s=0$, that
\[
x^{t-s(q-1)}\equiv\sum_{l=1}^{d+1}(-1)^{l+1}{d+1 \choose l}\cdot x^{t-(l+s)(q-1)}\mod(x^{q}-x)^{d+1}
\]

Substituting in (\ref{eq:induction step x^t mod (x^q-x)^(d+1)}) we
get
\begin{eqnarray*}
x^{t} & \equiv & {d+s \choose s}\cdot\sum_{l=1}^{d+1}(-1)^{l+1}{d+1 \choose l}\cdot x^{t-(l+s)(q-1)}+\\
 &  & +\sum_{l=2}^{d+1}(-1)^{l+1}{d+s \choose l+s-1}{l+s-2 \choose s-1}x^{t-(l+s-1)(q-1)}=
\end{eqnarray*}
\[
=\sum_{l=1}^{d+1}(-1)^{l+1}\left({d+s \choose s}{d+1 \choose l}-{d+s \choose l+s}{l+s-1 \choose s-1}\right)x^{t-(l+s)(q-1)}
\]

Calculation yields
\[
{d+s \choose s}{d+1 \choose l}-{d+s \choose l+s}{l+s-1 \choose s-1}=
\]
\[
=\frac{(d+s)!(d+1)!}{s!d!l!(d+1-l)!}-\frac{(d+s)!(l+s-1)!}{(l+s)!(d-l)!l!(s-1)!}=
\]
\[
=\frac{(d+s)!\cdot(d+1)\cdot(l+s)}{(l+s)\cdot s!l!(d+1-l)!}-\frac{(d+s)!\cdot s\cdot(d+1-l)}{(l+s)\cdot(d+1-l)!l!s!}=
\]
\[
=\frac{(d+s)!}{(l+s)\cdot s!l!(d+1-l)!}\cdot\left((d+1)l+(d+1)s-(d+1)s+sl\right)=
\]
\[
=\frac{(d+s+1)!}{(l+s)\cdot s!(l-1)!(d+1-l)!}=
\]
\[
=\frac{(d+1+s)!}{(l+s)!(d+1-l)!}\cdot\frac{(l+s-1)!}{s!(l-1)!}={d+1+s \choose l+s}{l+s-1 \choose s}
\]
establishing the identity (\ref{eq:x^t mod (x^q-x)^(d+1)}).

It now follows from (\ref{eq:x^t mod (x^q-x)^(d+1)}), by letting
$t=j+l(q-1)$ and $s=l-d-1$, that
\begin{equation}
\begin{split}
&f(x)=\sum_{i=0}^{k}c_{i}x^{i}=\sum_{i=0}^{d}c_{i}x^{i}+\sum_{j=d+1}^{d+q-1}\sum_{l=0}^{\left\lfloor \frac{k-j}{q-1}\right\rfloor }c_{j+l(q-1)}x^{j+l(q-1)}\equiv \\
&\equiv\sum_{i=0}^{d}c_{i}x^{i}+ \\
&+\sum_{j=d+1}^{d+q-1}\left(\sum_{l=0}^{d}c_{j+l(q-1)}x^{j+l(q-1)}+\sum_{l=d+1}^{\left\lfloor \frac{k-j}{q-1}\right\rfloor }c_{j+l(q-1)}\sum_{r=1}^{d+1}\gamma_{r,l,d}\cdot x^{j-(r-d-1)(q-1)}\right)=\\
&=\sum_{i=0}^{d}c_{i}x^{i}+\\
&\sum_{j=d+1}^{d+q-1}\sum_{l=0}^{d}\left(c_{j+l(q-1)}+\sum_{m=d+1}^{\left\lfloor \frac{k-j}{q-1}\right\rfloor }\delta_{m,l,d}\cdot c_{j+m(q-1)}\right)x^{j+l(q-1)}\mod(x^{q}-x)^{d+1}
\end{split}
\end{equation}

where
\[
\gamma_{r,l,d}=(-1)^{r+1}{l \choose r+l-d-1}\cdot{r+l-d-2 \choose l-d-1}
\]

and
\[
\delta_{m,l,d}=(-1)^{d-l}{m \choose l}{m-l-1 \choose d-l}.
\]
As this is a polynomial of degree less than $q(d+1)$, and we know
that $f(x)\in(x^{q}-x)^{d+1}\cdot C[x]+\varpi_{C}^{n}\cdot\mathcal{O}_{C}[x]$,
it follows that it must lie in $\varpi_{C}^{n}\cdot\mathcal{O}_{C}[x]$.

In particular, $c_{i}\in\varpi_{C}^{n}\mathcal{O}_{C}$ for all $0\le i\le d$,
and looking at the coefficient of $x^{j+d(q-1)}$ yields (\ref{eq:coefficients relation modulo (x^q-x)^(d+1)}),
as claimed.
\end{proof}
\begin{lem}
\label{subsec:Lemma divisibility of binomial coefficient} Let $q$
be a power of a prime number $p$. Let $k=d\cdot q+r$ be such that
$d<q$ and $0\le r<q-d$. Then for any $0\le i\le r$, any $0\le j\le d$
and any $j+1\le l\le d$, one has $p\mid{k-i \choose k-j-l(q-1)}$.
\end{lem}
\begin{proof}
Since $i\le r<q$, we know that $0\le r-i<q$ and $d<q$, so that
$k-i=d\cdot q+(r-i)$ is the base $q$ representation of $k-i$.

Since $j+1\le l\le d$, one has $1\le r+1\le r+l-j\le r+l\le r+d<q$
and it follows that
\[
k-j-l(q-1)=d\cdot q+r-l\cdot q+l-j=(d-l)\cdot q+(r+l-j)
\]
is the base $q$ representation of $k-j-l(q-1)$.

Finally, by Kummer's Theorem on binomial coefficients, as for any
$l\ge j+1$ and any $i,j\ge0$,
\[
r+l-j\ge r+1>r\ge r-i
\]
there is at least one digit in the base $p$ representation of $r+l-j$,
which is larger than the corresponding one in the base $p$ representation
of $r-i$, hence

\[
p\mid{k-i \choose k-j-l(q-1)}
\]
establishing the result.
\end{proof}
\begin{lem}
\label{lem:nonsingular matrix}Let $a\in\mathbb{Z}$. The matrix $A=A_{m}(a)\in\mathbb{Z}^{m\times m}$
with entries $(A_{li})_{l,i=1}^{m}={a+l \choose i-1}$ satisfies $\det A=1$.
\end{lem}
\begin{proof}
We prove it by induction on $m$. For $m=1$, this is the matrix $(1)$,
which is nonsingular.

Note that for any $2\le l\le m$, and any $1\le i\le m$, one has
\[
{a+l \choose i-1}-{a+l-1 \choose i-1}={a+l-1 \choose i-2}
\]

where ${k \choose -1}=0$.

Therefore, subtracting from each row its preceding row, we obtain
the matrix $B$, with $B_{1i}=A_{1i}$ for all $1\le i\le m$, and
$B_{li}={a+l-1 \choose i-2}$.

By the induction hypothesis, the matrix $(B_{li})_{l,i=2}^{m}$ is in fact $A_{m-1}(a)$, 
and $\det(B_{li})_{l,i=2}^{m}=1$. But, as $B_{l1}=0$
for all $l\ge2$ and $B_{11}=1$, it follows that $\det A=\det B=1$.
\end{proof}
\begin{cor}
\label{cor:nonsingular modified matrix}Let $a\in\mathbb{Z}$, $m\in\mathbb{N}$.
Let $t\in\{2,\ldots,m\}$. Consider the matrix $A\in\mathbb{Z}^{m\times m}$
with entries
\[
A_{li}=\begin{cases}
{a+l \choose i-1} & t\le i\le m\\
{a+l+1 \choose i-1} & 1\le i<t
\end{cases}\quad\forall l\in\{1,2,\ldots,m\}
\]
Then $\det A=1$.
\end{cor}
\begin{proof}
This matrix is obtained from the one in Lemma \ref{lem:nonsingular matrix}
by adding each of the first $t-2$ columns to its subsequent column,
since
\[
{a+l+1 \choose i-1}={a+l \choose i-1}+{a+l \choose i-2}
\]
As these operations do not affect the determinant, the result follows.
\end{proof}
\begin{cor}
\label{cor:nonsingular matrix large valuation}Let $k\in\mathbb{N}$.
Write $k=d\cdot q+r$, with $1\le d<p$, $0\le r<q$ and assume that
$d-1\le r$. Let $1\le m\le d$. Then the matrix $A\in\mathbb{F}_{p}^{m\times m}$
with entries $(A_{il})_{i,l=1}^{m}={k-i+1 \choose m+l(q-1)}$, is
nonsingular.
\end{cor}
\begin{proof}
For any $1\le l\le m$, we note that $m+l(q-1)=lq+(m-l)$, hence (as
$d<q$ and $r-i+1\ge r-d+1\ge0$) by Lucas' Theorem
\[
{k-i+1 \choose m+l(q-1)}={dq+r-i+1 \choose lq+(m-l)}\equiv{d \choose l}\cdot{r-i+1 \choose m-l}\mod p
\]
Since $1\le l\le d<p$, we get that the ${d \choose l}$ are nonzero
mod $p$, hence we can divide the $l$-th column by the appropriate
multiplier without affecting the singularity of $A$, call the resulting
matrix $B$.

Then $B_{il}={r-i+1 \choose m-l}$, which up to rearranging rows and
columns, is the matrix from Lemma \ref{lem:nonsingular matrix}, hence
nonsingular.
\end{proof}

\subsection{The case $v_{F}(a)\ge\left\lfloor \frac{k}{q}\right\rfloor $}

In this section, we will prove the following theorem, which will establish
$(iii)$ in Theorem \ref{thm:Main Theorem}.
\begin{thm}
Let $0\le k\le\min\left(p\cdot q-1,\frac{q^{2}}{2}\right)$. Assume
further that $k=dq+r$ with $d-1\le r<q$. Let $a\in\mathcal{O}_{C}$
be such that $v_{F}(a)\ge d$, and let $N\in\mathbb{Z}_{>0}$. There
exists a constant $\epsilon\in\mathbb{Z}_{\ge0}$ depending only on
$N,k,a$ such that for all $n\in\mathbb{Z}_{\ge0}$, and all $f\in ind_{KZ}^{G}\underline{\rho}_{k}$
\[
(T-a)(f)\in B_{N}+\varpi^{n}ind_{KZ}^{G}\underline{\rho}_{k}^{0}\Rightarrow f\in B_{N-1}+\varpi^{n-\epsilon}ind_{KZ}^{G}\underline{\rho}_{k}^{0}
\]
\end{thm}
\begin{proof}
As before, we may assume that $f=\sum_{m=0}^{M}f_{m}$ where $f_{m}\in S_{N+m}^{0}$,
and denote $f_{m}=0$ for $m>M$. Looking at $S_{N+m}$, we have the
equations

\[
T^{-}(f_{m+1})+T^{+}(f_{m-1})-af_{m}\in\varpi^{n}ind_{KZ}^{G}\underline{\rho}_{k}^{0}\quad1\le m\le M+1
\]

We shall prove the theorem with $\epsilon=d$.

Assume, by descending induction on $m$, that $f_{m},f_{m+1}\in\varpi^{n-d}ind_{KZ}^{G}\underline{\rho}_{k}^{0}$.
We will show that $f_{m-1}\in\varpi^{n-d}ind_{KZ}^{G}\underline{\rho}_{k}^{0}$.

By the above equations, we immediately obtain from (\ref{eq:formula for the coefficients})
(note that $af_{m}\in\varpi^{n}ind_{KZ}^{G}\underline{\rho}_{k}^{0}$,
since $v_{F}(a)\ge d$)
\begin{equation}
\sum_{i=j}^{k}{i \choose j}c_{i,\mu}^{m-1}[\lambda]^{i-j}\in\varpi^{n-d-j}\mathcal{O}_{C}\label{eq:first approximation for divisibility}
\end{equation}
for all $\mu\in I_{m-1}$, all $\lambda\in\mathbb{F}_{q}$, and all
$0\le j\le d$.

By Lemma \ref{subsec:Lemma divisibility of polynomial coefficients},
it follows that for all $i$, $c_{i,\mu}^{m-1}\in\varpi^{n-2d}\mathcal{O}_{C}$.

Next, for any $1\le j\le d$, consider the formulas for $C_{j+l(q-1),\mu}^{m}$
for any $1\le l\le d$. Note that $j+l(q-1)\le d+d(q-1)=dq\le k$.

Since $k\le q^{2}/2$, one has
\[
d=\left[\frac{k}{q}\right]\le\frac{k}{q}\le\frac{q}{2}\Rightarrow2d\le q
\]
so that $n-2d+q\ge n$.

Therefore, we get that
\[
\varpi^{j+l(q-1)}{i \choose j+l(q-1)}c_{i,\mu}^{m-1}\in\varpi^{q}c_{i,\mu}^{m-1}\mathcal{O}_{C}\subseteq\varpi^{n-2d+q}\mathcal{O}_{C}\subseteq\varpi^{n}\mathcal{O}_{C}
\]
for all $j,l$. Since for $i\le k-d$, $\varpi^{d}\mid\varpi^{k-i}$
and $c_{i,\mu+\varpi^{m}[\lambda]}^{m+1}\in\varpi^{n-d}\mathcal{O}_{C}$,
it follows that
\begin{equation}
\begin{split}
&C_{j+l(q-1),\mu}^{m}\equiv \\
&\equiv \sum_{i=k-d+1}^{k}\varpi^{k-i}{i \choose j+l(q-1)}\sum_{\lambda\in\kappa_{F}}c_{i,\mu+\varpi^{m}[\lambda]}^{m+1}[\lambda]^{i-j-l(q-1)} \equiv \\
&\equiv0\mod\varpi^{n}\mathcal{O}_{C}
\end{split}
\end{equation}
Since $k=d\cdot q+r$, with $r\ge d$, we see that $k-d+1-d\cdot q=r+1-d\ge1$,
showing that for any $1\le l\le d$, any $k-d+1\le i\le k$, we get
$i-j-l(q-1)\ge1$, hence for any $\lambda\in\kappa_{F}$, $[\lambda]^{i-j-l(q-1)}=[\lambda]^{i-j}$.
(Had $i-j-l(q-1)$ been $0$, this is violated when $\lambda=0$!).

By the induction hypothesis, we know that $c_{k-i,\mu}^{m+1}\in\varpi^{n-d}\mathcal{O}_{C}$.
Write, for $0\le i\le d-1$ and $1\le j\le d$, $\sum_{\lambda\in\mathbb{F}_{q}}c_{k-i,\mu+\pi^{m}[\lambda]}^{m+1}[\lambda]^{k-i-j}=\varpi^{n-d}\cdot x_{ij}$
for some $x_{ij}\in\mathcal{O}_{C}$. Then the above equations for
$1\le l\le d$ yield
\begin{equation}
\sum_{i=0}^{d-1}\varpi^{i}{k-i \choose j+l(q-1)}\cdot x_{ij}\equiv0\mod\varpi^{d}\label{eq:T^(-) bootstrapping formula}
\end{equation}

Let us prove that that $x_{ij}\in\varpi^{j-i}\mathcal{O}_{C}$ for
all $1\le j\le d$ and all $0\le i\le j$. Note that for $i=j$, it
is trivial, so we will prove it for $0\le i\le j-1$.

Indeed, fix $j$. Then, looking modulo $\varpi^{j}$, and setting
$y_{ij}=\varpi^{i}x_{ij}$, one obtains the equations (for all $0\le i\le j-1$
and all $1\le l\le j$)
\[
\sum_{i=0}^{j-1}{k-i \choose j+l(q-1)}\cdot y_{ij}\equiv0\mod\varpi^{j}.
\]

By Corollary \ref{cor:nonsingular matrix large valuation}, with $m=j$,
we see that the matrix of coefficients here is nonsingular modulo
$p$, hence also invertible modulo $\varpi^{j}$, and it follows that
$y_{ij}\in\varpi^{j}\mathcal{O}_{C}$ for all $0\le i\le j-1$. But
this precisely means that
\[
x_{ij}=\varpi^{-i}y_{ij}\in\varpi^{j-i}\mathcal{O}_{C}
\]
as claimed.

Therefore,
\[
\varpi^{i}\cdot\sum_{\lambda\in\mathbb{F}_{q}}c_{k-i,\mu+\varpi^{m}[\lambda]}^{m+1}\cdot[\lambda]^{k-i-j}=\varpi^{i}\cdot\varpi^{n-d}x_{ij}\in\varpi^{n-d+j}\mathcal{O}_{C}
\]
Considering now the formulas for $C_{j,\mu}^{m}$, with $1\le j\le d$,
we get
\[
\sum_{i=j}^{k}{i \choose j}c_{i,\mu}^{m-1}[\lambda]^{i-j}\in\varpi^{n-d}\mathcal{O}_{C}.
\]

This also holds when $j=0$ trivially as a conequence of (\ref{eq:first approximation for divisibility}).

Hence, applying once more Lemma \ref{subsec:Lemma divisibility of polynomial coefficients},
\[
c_{i,\mu}^{m-1}\in\varpi^{n-d}\mathcal{O}_{C}
\]
as claimed. Therefore, in this case, taking $\epsilon=d$ suffices.
\end{proof}

\subsection{The case $0<v_{F}(a)\le e$}

In this subsection, we will prove the following theorem. Since the
case $v_{F}(a)=0$ is covered by \cite[Prop. 4.10]{de2013existence},
it establishes $(i)$ and $(ii)$ in Theorem \ref{thm:Main Theorem},
for that case.
\begin{thm}
\label{subsec:Theorem less than 1} Let $0\le k\le q^{2}/2$. Assume
further that $k=dq+r$ with $0\le r<q-d$. Let $a\in\mathcal{O}_{C}$
be such that $0<v_{F}(a)\le e$. Assume either that $0<v_{F}(a)\le1$
or that $2v_{F}(a)-1\le r$. Then $(k,a)$ is separated.
\end{thm}
We prove the theorem by considering two cases.

We shall first prove the case where $\max(2v_{F}(a)-1,1)\le r$, and
then the case $r=0,v_{F}(a)\le1$.

Unfortunately, we have not been able to provide a proof for the case
$0\le r<2v_{F}(a)-1$.
\begin{proof}
Let $f\in ind_{KZ}^{G}\underline{\rho}_{k}$ be such that $(T-a)f\in B_{N}+\varpi^{n}ind_{KZ}^{G}\underline{\rho}_{k}^{0}$.
We may assume that $f=\sum_{m=0}^{M}f_{m}$ where $f_{m}\in S_{N+m}$,
and denote $f_{m}=0$ for $m>M$. Looking at $S_{N+m}$, we have the
equations
\begin{equation}
T^{-}(f_{m+1})+T^{+}(f_{m-1})-af_{m}\in\varpi^{n}ind_{KZ}^{G}\underline{\rho}_{\underline{k}}^{0}\quad1\le m\le M+1\label{eq:Sphere equations}
\end{equation}
Our proof will be based on descending induction on $m$, showing that
if $f_{m}$, $f_{m+1}$ are highly divisible, so must be $f_{m-1}$.

We will initially obtain some bound for the valuation of $f_{m-1}$
using $f_{m}$ and $f_{m+1}$, and then we will use that initial bound
to bootstrap and obtain better bounds on the valuation of $f_{m}$,$f_{m+1}$
and, in turn, $f_{m-1}$.

Moreover, we may assume that $f_{m}\in S_{N+m}^{0}$, using $G$-equivariance.

We refer the reader to the definition of the coefficients $c_{\underline{j},\mu}^{m}$
in Corollary \ref{cor:definition of coefficients}, and to formula
(\ref{eq:formula for the coefficients}).

As under our assumptions $\left|S^{+}\right|=1$, we will usually
replace the multi-index notation $\underline{j}$ by $j=j_{\sigma}$.

The idea of this part of the proof is as follows - the contribution
from the $T^{+}$ part (the inner vertex) has high valuation when
$j$ is large, while the contribution from the $T^{-}$ part (the
outer vertices) has high valuation when $j$ is small.

Let us introduce the statements $\mathscr{A}_{m},\mathscr{B}_{m},\mathscr{C}_{m},\mathscr{D}_{m}$
for the rest of the proof.

The assumptions $\mathscr{A}_{m}$ are made to ensure that for small
values of $j$, the contribution from $T^{+}$ is of high enough valuation,
hence we can infer something about its preimage (by the previous Lemmata).
These give us the initial bound for the valuation of $f_{m-1}$.

In the bootstrapping part, this bound shows that for large values
of $j$, the main contribution comes from $T^{-}$, whence we must
use $\mathscr{B}_{m}$ in order to obtain better bounds on the valuation
of $f_{m}$. These bounds for large values of $j$ can improve our
bounds for small values of $j$ by using the assumption $\mathscr{C}_{m}$,
which is a linear relation involving one small value of $j$, while
all the others are large.

Finally, this is used to obtain a better bound on the valuation of
$f_{m-1}$, establishing the theorem.
\[
\mathscr{A}_{m}:\quad c_{j,\mu}^{m}\in\frac{\varpi^{n-j}}{a}\cdot\mathcal{O}_{C}\quad\forall0\le j\le d,\quad c_{i,\mu}^{m}\in\frac{\varpi^{n-d}}{a}\cdot\mathcal{O}_{C}\quad\forall0\le i\le k\quad\forall\mu\in I_{m}
\]
\[
\mathscr{B}_{m}:\quad c_{k-j,\mu}^{m}\in\frac{\varpi^{n-j}}{a}\cdot\mathcal{O}_{C}\quad\forall0\le j\le d,\quad c_{i,\mu}^{m}\in\frac{\varpi^{n-d}}{a}\cdot\mathcal{O}_{C}\  \forall0\le i\le k \ \forall\mu\in I_{m}
\]
\[
\mathscr{C}_{m}:\quad\sum_{s=j}^{\left\lfloor \frac{k-i}{q-1}\right\rfloor }{s \choose j}\cdot c_{i+s(q-1),\mu}^{m}\in\frac{\varpi^{n-j}}{a}\cdot\mathcal{O}_{C}\quad\forall j+1\le i\le j+q-1,\quad\forall0\le j\le d
\]
\[
\mathscr{D}_{m}:\quad c_{i,\mu}^{m}\in\frac{\varpi^{n}}{a^{2}}\cdot\mathcal{O}_{C}\quad\forall0\le i\le k
\]
Assume, by descending induction on $m$, that $\mathscr{A}_{m},\mathscr{B}_{m},\mathscr{B}_{m+1},\mathscr{C}_{m}$
hold for all $\mu,\lambda$.

Note that, as $f_{M+1}=f_{M+2}=0$, they trivially hold for $m=M+1$.
We will prove that $\mathscr{A}_{m-1},\mathscr{B}_{m-1},\mathscr{B}_{m},\mathscr{C}_{m-1}$
hold.

For this, we make use of the subsequent Lemma \ref{lem:Main Lemma for v(a) < e case}.

We assume $\mathscr{A}_{m},\mathscr{B}_{m},\mathscr{B}_{m+1},\mathscr{C}_{m}$,
hence by Lemma \ref{lem:Main Lemma for v(a) < e case}, we know that

$\mathscr{A}_{m-1},\mathscr{C}_{m-1},\mathscr{D}_{m}$ also hold.

It remains to show that $\mathscr{B}_{m-1}$ holds. In fact, we need
only to show that $c_{k-j,\mu}^{m-1}\in\frac{\varpi^{n-j}}{a}\mathcal{O}_{C}$
for all $0\le j\le d$.

Note that since $\mathscr{B}_{m}$ holds, by applying Lemma \ref{lem:Main Lemma for v(a) < e case}
to $m-1$, we see that $\mathscr{A}_{m-2},\mathscr{C}_{m-2}$ hold
as well, and so does $\mathscr{D}_{m-1}$.

Next, we see from $\mathscr{D}_{m-1}$ that we have $c_{k-j,\mu}^{m-1}\in\frac{\varpi^{n}}{a^{2}}\mathcal{O}_{C}\subseteq\frac{\varpi^{n-j}}{a}\mathcal{O}_{C}$
for all $v_{F}(a)\le j\le d$, which we get ``for free''. Therefore,
it remains to show that $c_{k-j,\mu}^{m-1}\in\frac{\varpi^{n-j}}{a}\cdot\mathcal{O}_{C}$
for all $0\le j<\min(v_{F}(a),d)$.

Fix some $0\le j<\min(v_{F}(a),d)$.

Now, since by Lemma \ref{subsec:Lemma divisibility of binomial coefficient},
$p\mid{i \choose k-j-l(q-1)}$ for all $k-2v_{F}(a)<i\le k$ and all
$j+1\le l\le d$ (here we use $2v_{F}(a)-1\le r<q-d$), and by $\mathscr{B}_{m}$,
$c_{i,\mu}^{m}\in\frac{\varpi^{n-k+i}}{a}\mathcal{O}_{C}$ for all
$k-2v_{F}(a)<i\le k$, we get (as $\varpi^{e}\mid p$) that
\begin{equation}
\varpi^{k-i}\cdot{i \choose k-j-l(q-1)}\cdot c_{i,\mu}^{m}\in\varpi^{k-i+e}\cdot\frac{\varpi^{n-k+i}}{a}\cdot\mathcal{O}_{C}=\frac{\varpi^{n+e}}{a}\cdot\mathcal{O}_{C}\subseteq\varpi^{n}\mathcal{O}_{C}\label{eq:large indices part of T- is divisible}
\end{equation}
where the last inclusion follows from $v_{F}(a)\le e$.

Furthermore, since we have shown $\mathscr{D}_{m}$, we know that
$c_{i,\mu}^{m}\in\frac{\varpi^{n}}{a^{2}}\mathcal{O}_{C}=\varpi^{n-2v_{F}(a)}\mathcal{O}_{C}$
for all $0\le i\le k$, hence for $i\le k-2v_{F}(a)$, we get
\begin{equation}
\varpi^{k-i}\cdot c_{i,\mu}^{m}\in\varpi^{2v_{F}(a)}\cdot\varpi^{n-2v_{F}(a)}\mathcal{O}_{C}=\varpi^{n}\mathcal{O}_{C}.\label{eq:small indices part of T- is divisible}
\end{equation}

At this point we make use of the hypothesis (\ref{eq:Sphere equations}).

It then follows from equation (\ref{eq:formula for the coefficients})
for $C_{k-j-l(q-1)}^{m-1}$, and equations (\ref{eq:large indices part of T- is divisible}),
(\ref{eq:small indices part of T- is divisible}) that for all $\mu\in I_{m-1}$
\[
\varpi^{k-j-l(q-1)}\cdot\sum_{i=k-j-l(q-1)}^{k}{i \choose k-j-l(q-1)}\cdot c_{i,[\mu]_{m-2}}^{m-2}[-\lambda_{\mu}]^{i-k+j+l(q-1)}-
\]
\[
-a\cdot c_{k-j-l(q-1),\mu}^{m-1}\in\varpi^{n}\mathcal{O}_{C}.
\]
But recall that $l\le d$, so that
\[
k-j-l(q-1)=(d-l)\cdot(q-1)+(r+d-j)\ge r+d-j\ge d+\max(1,2v_{F}(a)-1)-j
\]
where in the last inequality we use our assumption that $r\ge1$.

Since we have established $\mathscr{A}_{m-2}$, we know that $c_{i,\mu}^{m-2}\in\frac{\varpi^{n-d}}{a}\cdot\mathcal{O}_{C}$,
hence $\varpi^{k-j-l(q-1)}\cdot c_{i,\mu}^{m-2}\in\frac{\varpi^{n+\max(1,2v_{F}(a)-1)-j}}{a}\cdot\mathcal{O}_{C}\subseteq\varpi^{n-j}\mathcal{O}_{C}$.

Therefore, we obtain that $a\cdot c_{k-j-l(q-1),\mu}^{m-1}\in\varpi^{n-j}\mathcal{O}_{C}$,
hence
\begin{equation}
c_{k-j-l(q-1),\mu}^{m-1}\in\frac{\varpi^{n-j}}{a}\mathcal{O}_{C}\quad\forall j+1\le l\le d.\label{eq:divisibility of coeffs in multiples of q-1}
\end{equation}

We shall now use $\mathscr{C}_{m-1}$ to infer from the divisibility
of these coefficients, the divisibility of the coefficient $c_{k-j,\mu}^{m-1}$
by $\frac{\varpi^{n-j}}{a}$ as desired. This shall be done as follows.

Let $i$ be the unique integer satisfying $j+1\le i\le j+q-1$ such
that $i\equiv k-j\mod(q-1)$, and let $l_{0}=\left\lfloor \frac{k-i}{q-1}\right\rfloor $,
so that $k-j=i+l_{0}(q-1)$. (Recall that $k-j+q-1\ge k-d+q-1>k)$.

If $i<q$, we let $A\in\mathbb{Z}^{(j+1)\times(j+1)}$ be the matrix
with entries $A_{tl}={l_{0}-l \choose t}_{t,l=0}^{j}$.

If $i\ge q$, we let $A$ be the matrix with entries
\[
A_{tl}=\begin{cases}
{l_{0}-l \choose t} & i-q<t\le j\\
{l_{0}-l+1 \choose t} & 0\le t\le i-q
\end{cases}\quad\forall l\in\{0,1,2,\ldots,j\}
\]

In each of the cases, $A\in GL_{j+1}(\mathbb{Z})$, either by Lemma
\ref{lem:nonsingular matrix} or by Corollary \ref{cor:nonsingular modified matrix}.

Therefore, there exists a non-trivial $\mathbb{Z}$-linear combination
of its rows, some $\alpha_{t}\in\mathbb{Z}$ such that for all $0\le l\le j$
\begin{equation}
\sum_{t=0}^{j}\alpha_{t}A_{tl}=\delta_{l,0}.\label{eq:linear combination of Cm}
\end{equation}

For $t>i-q$, substituting in $\mathscr{C}_{m-1}$ the value $t$
for $j$, we obtain for all $\mu\in I_{m-1}$
\[
\Xi_{t}:=\sum_{s=t}^{l_{0}}{s \choose t}\cdot c_{i+s(q-1),\mu}^{m-1}\in\frac{\varpi^{n-t}}{a}\cdot\mathcal{O}_{C}\subseteq\frac{\varpi^{n-j}}{a}\cdot\mathcal{O}_{C}.
\]

Note that indeed $t+1\le j+1\le i\le t+q-1$, as required.

For $0\le t\le i-q$, substituting in $\mathscr{C}_{m-1}$ the value
$t$ for $j$ and the value $i-(q-1)$ for $i$, we obtain for all
$\mu\in I_{m-1}$
\begin{equation}
\begin{split}
&\Xi_{t}:=\sum_{s=t-1}^{l_{0}}{s+1 \choose t}\cdot c_{i+s(q-1),\mu}^{m-1}= \\
&=\sum_{s=t}^{l_{0}}{s \choose t}\cdot c_{i+(s-1)(q-1),\mu}^{m-1}\in\frac{\varpi^{n-t}}{a}\cdot\mathcal{O}_{C}\subseteq\frac{\varpi^{n-j}}{a}\cdot\mathcal{O}_{C}.
\end{split}
\end{equation}

Note that indeed $t+1\le i-(q-1)\le j\le d-1\le q-1\le t+q-1$, as
required.

Considering the linear combination $\sum_{t=0}^{j}\alpha_{t}\Xi_{t}$,
we see that
\[
\sum_{t=0}^{i-q}\sum_{s=t-1}^{l_{0}}\alpha_{t}{s+1 \choose t}\cdot c_{i+s(q-1),\mu}^{m-1}+\sum_{t=i-q+1}^{j}\sum_{s=t}^{l_{0}}\alpha_{t}{s \choose t}\cdot c_{i+s(q-1),\mu}^{m-1}=
\]
\[
=\sum_{t=0}^{j}\alpha_{t}\Xi_{t}\in\frac{\varpi^{n-j}}{a}\cdot\mathcal{O}_{C}
\]
which, reindexing, is the same as
\begin{equation}
\begin{split}
&\sum_{l=0}^{l_{0}+1}\left(\sum_{t=0}^{i-q}\alpha_{t}{l_{0}-l+1 \choose t}\right)\cdot c_{k-j-l(q-1),\mu}^{m-1}+\\
&+\sum_{l=0}^{l_{0}}\left(\sum_{t=i-q+1}^{j}\alpha_{t}{l_{0}-l \choose t}\right)\cdot c_{k-j-l(q-1),\mu}^{m-1}=\\
&=\sum_{l=0}^{l_{0}+1}\left(\sum_{t=0}^{i-q}\alpha_{t}{l_{0}-l+1 \choose t}\right)\cdot c_{i+l_{0}(q-1)-l(q-1),\mu}^{m-1}+\\
&+\sum_{l=0}^{l_{0}}\left(\sum_{t=i-q+1}^{j}\alpha_{t}{l_{0}-l \choose t}\right)\cdot c_{i+l_{0}(q-1)-l(q-1),\mu}^{m-1}
\label{eq:expanded linear combination of Cm}
\end{split}
\end{equation}
which lies in $\frac{\varpi^{n-j}}{a}\cdot\mathcal{O}_{C}$.

Since we assumed that $r<q-d$ we have
\[
k-j-(d+1)(q-1)\le k-(d+1)(q-1)=d\cdot q+r-(dq+q-d-1)=
\]
\[
=r-(q-d-1)\le0<1\le j+1\le i=k-j-l_{0}(q-1)
\]
showing that $l_{0}\le d$, hence for every $j+1\le l\le l_{0}$,
by (\ref{eq:divisibility of coeffs in multiples of q-1}) we have
$c_{k-j-l(q-1),\mu}^{m-1}\in\frac{\varpi^{n-j}}{a}\mathcal{O}_{C}$,
so that (\ref{eq:expanded linear combination of Cm}) yields
\[
\sum_{l=0}^{j}\left(\sum_{t=0}^{j}\alpha_{t}A_{tl}\right)\cdot c_{k-j-l(q-1),\mu}^{m-1}=
\]
\[
=\sum_{l=0}^{j}\left(\sum_{t=0}^{i-q}\alpha_{t}{l_{0}-l+1 \choose t}+\sum_{t=i-q+1}^{j}\alpha_{t}{l_{0}-l \choose t}\right)\cdot c_{k-j-l(q-1),\mu}^{m-1}\in\frac{\varpi^{n-j}}{a}\cdot\mathcal{O}_{C}.
\]
Now we apply (\ref{eq:linear combination of Cm}) to see that this
is no more than $c_{k-j,\mu}^{m-1}\in\frac{\varpi^{n-j}}{a}\cdot\mathcal{O}_{C}$,
as wanted. This establishes $\mathscr{B}_{m-1}$.

At this point, we have established $\mathscr{A}_{m-1},\mathscr{B}_{m-1},\mathscr{B}_{m},\mathscr{C}_{m-1}$
from 

$\mathscr{A}_{m},\mathscr{B}_{m},\mathscr{B}_{m+1},\mathscr{C}_{m}$.
By descending induction, this shows that 

$\mathscr{A}_{m},\mathscr{B}_{m},\mathscr{B}_{m+1},\mathscr{C}_{m}$
hold for all $m$.

In particular, considering for example $\mathscr{A}_{m}$, we see
that for any $m$, $c_{i,\mu}^{m}\in\frac{\varpi^{n-d}}{a}\cdot\mathcal{O}_{C}$
for all $0\le i\le k$, showing that $f_{m}\in\varpi^{n-(d+v_{F}(a))}\cdot\mathcal{O}_{C}$
for all $m$.

Thus, we have shown that if $(T-a)f\in B_{N}+\varpi^{n}ind_{KZ}^{G}\underline{\rho}_{k}^{0}$,
then $f\in B_{N-1}+\varpi^{n-(d+v_{F}(a))}\cdot\mathcal{O}_{C}$.

Therefore, in the case $\max(2v_{F}(a)-1,1)\le r$, taking $\epsilon=d+v_{F}(a)$
suffices in order to show that $(k,a)$ is separated.
\end{proof}
\begin{lem}
\label{lem:Main Lemma for v(a) < e case}Assume that for some $m$,
$\mathscr{A}_{m},\mathscr{B}_{m+1},\mathscr{C}_{m}$ hold. 

Then $\mathscr{A}_{m-1},\mathscr{C}_{m-1},\mathscr{D}_{m}$ hold as well.
\end{lem}
\begin{proof}
From (\ref{eq:Sphere equations}) and (\ref{eq:formula for the coefficients})
we see that for any $0\le j\le d$
\begin{equation}
\begin{split}
C_{j,\mu}^{m}=\sum_{i=j}^{k}\varpi^{k-i}{i \choose j}\sum_{\lambda\in\kappa_{F}}c_{i,\mu+\varpi^{m}[\lambda]}^{m+1}[\lambda]^{i-j}+\\
+\varpi^{j}\sum_{i=j}^{k}c_{i,[\mu]_{m-1}}^{m-1}{i \choose j}(-\lambda_{\mu})^{i-j}-ac_{j,\mu}^{m}\in\varpi^{n}\mathcal{O}_{C}\label{eq:valuation of spherecoeffs}
\end{split}
\end{equation}
where $\lambda_{\mu}=\frac{\mu-[\mu]_{m-1}}{\varpi^{m-1}}$.

By the hypothesis $\mathscr{B}_{m+1}$, for any $k-d<i\le k$ (and
any $\mu$), we have $c_{i,\mu}^{m+1}\in\frac{\varpi^{n-k+i}}{a}\cdot\mathcal{O}_{C}$,
hence $\varpi^{k-i}\cdot c_{i,\mu}^{m+1}\in\frac{\varpi^{n}}{a}\cdot\mathcal{O}_{C}$.

Also, for any $0\le i\le k-d$, by $\mathscr{B}_{m+1}$, we have $c_{i,\mu}^{m+1}\in\frac{\varpi^{n-d}}{a}\cdot\mathcal{O}_{C}$,
hence $\varpi^{k-i}\cdot c_{i,\mu}^{m+1}\in\frac{\varpi^{d}\cdot\varpi^{n-d}}{a}\cdot\mathcal{O}_{C}=\frac{\varpi^{n}}{a}\cdot\mathcal{O}_{C}$.

We conclude that for any $0\le i\le k$, one has
\begin{equation}
\varpi^{k-i}\cdot c_{i,\mu}^{m+1}\in\frac{\varpi^{n}}{a}\cdot\mathcal{O}_{C}.\label{eq:valuation of m+1 coeffs}
\end{equation}

This implies that the first sum in (\ref{eq:valuation of spherecoeffs})
lies in $\frac{\varpi^{n}}{a}\cdot\mathcal{O}_{C}$, hence
\begin{equation}
\varpi^{j}\sum_{i=j}^{k}c_{i,[\mu]_{m-1}}^{m-1}{i \choose j}(-\lambda_{\mu})^{i-j}-ac_{j,\mu}^{m}\in\frac{\varpi^{n}}{a}\mathcal{O}_{C}\label{eq:valuation of spherecoeffs-1}
\end{equation}

Furthermore, for any $0\le j\le d$, by $\mathscr{A}_{m}$, we know
that $c_{j,\mu}^{m}\in\frac{\varpi^{n-j}}{a}\cdot\mathcal{O}_{C}$,
hence
\begin{equation}
ac_{j,\mu}^{m}\in\varpi^{n-j}\mathcal{O}_{C}\label{eq:valuation of m coeffs}
\end{equation}

$\bullet$ If $v_{F}(a)\le j$, we see that $\varpi^{n}/a\in\varpi^{n-j}\mathcal{O}_{C}$,
so we get from (\ref{eq:valuation of m+1 coeffs}), (\ref{eq:valuation of m coeffs}),
and (\ref{eq:valuation of spherecoeffs}) that
\begin{equation}
\varpi^{j}\sum_{i=j}^{k}{i \choose j}c_{i,\mu}^{m-1}[\lambda]^{i-j}\in\varpi^{n-j}\mathcal{O}_{C}\Rightarrow\sum_{i=j}^{k}{i \choose j}c_{i,\mu}^{m-1}[\lambda]^{i-j}\in\varpi^{n-2j}\mathcal{O}_{C}\label{eq:initial estimate on first coeffs}
\end{equation}
for all $v_{F}(a)\le j\le d$, for all $\mu\in I_{m-1}$ and for all
$\lambda\in\kappa_{F}$.

$\bullet$ If $j\le v_{F}(a)$, we see that $\varpi^{n-j}\in\frac{\varpi^{n}}{a}\cdot\mathcal{O}_{C}$,
so we get from (\ref{eq:valuation of m+1 coeffs}), (\ref{eq:valuation of m coeffs}),
and (\ref{eq:valuation of spherecoeffs}) that
\begin{equation}
\varpi^{j}\sum_{i=j}^{k}{i \choose j}c_{i,\mu}^{m-1}[\lambda]^{i-j}\in\frac{\varpi^{n}}{a}\cdot\mathcal{O}_{C}\Rightarrow\sum_{i=j}^{k}{i \choose j}c_{i,\mu}^{m-1}[\lambda]^{i-j}\in\frac{\varpi^{n-j}}{a}\cdot\mathcal{O}_{C}.\label{eq:initial estimate on first coeffs-1}
\end{equation}

In particular, by Lemma \ref{subsec:Lemma divisibility of polynomial coefficients},
we see that if $v_{F}(a)\le d$, then $c_{i,\mu}^{m-1}\in\varpi^{n-2d}\mathcal{O}_{C}$
for all $0\le i\le k$, and if $v_{F}(a)\ge d$, then $c_{i,\mu}^{m-1}\in\frac{\varpi^{n-d}}{a}\cdot\mathcal{O}_{C}$
for all $0\le i\le k$.

Substituting $\lambda=0$ in (\ref{eq:initial estimate on first coeffs-1})
we get $c_{j,\mu}^{m-1}\in\frac{\varpi^{n-j}}{a}\mathcal{O}_{C}$.

Therefore, if $v_{F}(a)\ge d$, we have already established $\mathscr{A}_{m-1}$.
In this case, since $\frac{\varpi^{n-d}}{a}\in\frac{\varpi^{n}}{a^{2}}\cdot\mathcal{O}_{C}$,
$\mathscr{D}_{m}$ trivially holds.

If $v_{F}(a)<d$, we consider the coefficients $C_{2d,\mu}^{m},C_{2d+1,\mu}^{m},\ldots,C_{k,\mu}^{m}$.
By (\ref{eq:valuation of spherecoeffs}) and (\ref{eq:valuation of m+1 coeffs}),
using the fact that $c_{i,\mu}^{m-1}\in\varpi^{n-2d}\mathcal{O}_{C}$
for all $i$, we get that $ac_{j,\mu}^{m}\in\frac{\varpi^{n}}{a}\mathcal{O}_{C}$
for all $j\ge2d$.

In particular, since, by assumption, $q\ge2k/q\ge2d$, we get that
for any $1\le j\le2d-1$ and any $1\le l$,
\[
j+l(q-1)\ge1+(q-1)=q\ge2d
\]
hence $c_{j+l(q-1),\mu}^{m}\in\frac{\varpi^{n}}{a^{2}}\mathcal{O}_{C}$.

By the assumption $\mathscr{C}_{m}$ (substituting $j$ for $i$ and
$0$ for $j$), it follows also that $c_{j,\mu}^{m}\in\frac{\varpi^{n}}{a^{2}}\mathcal{O}_{C}$.
Therefore $ac_{j,\mu}^{m}\in\frac{\varpi^{n}}{a}\mathcal{O}_{C}$
for all $1\le j\le2d-1$, hence for all $0\le j\le k$, establishing
$\mathscr{D}_{m}$. Note that the case $j=0$ is given by $\mathscr{A}_{m}$.

We may now consider once more the equations for $C_{1,\mu}^{m},\ldots,C_{d,\mu}^{m}$,
and get from (\ref{eq:valuation of m+1 coeffs}), (\ref{eq:valuation of spherecoeffs})
and $\mathscr{D}_{m}$ that $\forall1\le j\le d$
\[
\varpi^{j}\sum_{i=j}^{k}{i \choose j}c_{i,\mu}^{m-1}[\lambda]^{i-j}\in\frac{\varpi^{n}}{a}\mathcal{O}_{C}\Rightarrow\sum_{i=j}^{k}{i \choose j}c_{i,\mu}^{m-1}[\lambda]^{i-j}\in\frac{\varpi^{n-j}}{a}\mathcal{O}_{C}.
\]
When $j=0$, this holds by (\ref{eq:initial estimate on first coeffs-1}).
By Lemma \ref{subsec:Lemma divisibility of polynomial coefficients},
it follows that $c_{i,\mu}^{m-1}\in\frac{\varpi^{n-d}}{a}\mathcal{O}_{C}$
for all $i$. Also, it shows that $c_{j,\mu}^{m-1}\in\frac{\varpi^{n-j}}{a}\mathcal{O}_{C}$
for $1\le j\le d$, by substituting $\lambda=0$. Therefore, we have
established $\mathscr{A}_{m-1}$ in this case as well.

Finally, for any $0\le j\le d$, and for any $0\le t\le j$
\[
\sum_{i=t}^{k}{i \choose t}c_{i,\mu}^{m-1}[\lambda]^{i-t}\in\frac{\varpi^{n-t}}{a}\cdot\mathcal{O}_{C}\subseteq\frac{\varpi^{n-j}}{a}\cdot\mathcal{O}_{C}
\]
for all $\lambda\in\kappa_{F}$. Thus, by Lemma \ref{subsec:Lemma sum coeffs divisibility},
substituting $j$ for $d$ and $i$ for $j$, we get $\mathscr{C}_{m-1}$.
\end{proof}
We now consider the case $0<v_{F}(a)\le1$ and $r=0$, using a different
argument.
\begin{thm}
Let $k=dq$, and assume $1\le d<\frac{q}{2}$ (note that this excludes
$q=2$) . Let $a\in\mathcal{O}_{C}$ be such that $0<v_{F}(a)\le1$,
and let $N\in\mathbb{Z}_{>0}$. There exists a constant $\epsilon\in\mathbb{Z}_{\ge0}$
depending only on $N,k,a$ such that for all $n\in\mathbb{Z}_{\ge0}$,
and all $f\in ind_{KZ}^{G}\underline{\rho}_{k}^{0}$:
\[
(T-a)(f)\in B_{N}+\varpi^{n}ind_{KZ}^{G}\underline{\rho}_{k}^{0}\Rightarrow f\in B_{N-1}+\varpi^{n-\epsilon}ind_{KZ}^{G}\underline{\rho}_{k}^{0}
\]
\end{thm}
\begin{proof}
We may assume that $f=\sum_{m=0}^{M}f_{m}$ where $f_{m}\in S_{N+m}^{0}$,
and denote $f_{m}=0$ for $m>M$. Looking at $S_{N+m}$ , we have
the equations
\[
T^{-}(f_{m+1})+T^{+}(f_{m-1})-af_{m}\in\varpi^{n}ind_{KZ}^{G}\underline{\rho}_{k}^{0}\quad1\le m\le M+1
\]
Assume, by descending induction on $m$, that the following hold:
\begin{equation}
\begin{split}
&c_{k,\mu}^{m+1}\in\frac{\varpi^{n}}{a^{2}}\mathcal{O}_{C},\quad c_{k-j,\mu}^{m+1}\in\frac{\varpi^{n}}{a}\mathcal{O}_{C}\quad\forall0<j\le d \\
&c_{i,\mu}^{m+1}\in\frac{\varpi^{n-d}}{a}\mathcal{O}_{C}\quad\forall0\le i\le k\quad\forall\mu\in I_{m+1} \\
&c_{0,\mu}^{m}\in\frac{\varpi^{n}}{a}\mathcal{O}_{C},\quad\sum_{l=0}^{\left\lfloor \frac{k-j}{q-1}\right\rfloor }c_{j+l(q-1),\mu}^{m}\in\frac{\varpi^{n}}{a}\mathcal{O}_{C}\quad\forall1\le j\le d \\
&c_{i,\mu}^{m}\in\frac{\varpi^{n-d}}{a}\mathcal{O}_{C}\quad\forall0\le i\le k\quad\forall\mu\in I_{m} \\
&\sum_{\lambda\in\kappa_{F}}c_{k,\mu+\varpi^{m}[\lambda]}^{m+1}[\lambda]^{l}\in\frac{\varpi^{n}}{a}\mathcal{O}_{C},\quad\forall l\in\{0,1,2,\ldots,d,q-1\}\quad\forall\mu\in I_{m}
\label{eq:contribution of outer sums}
\end{split}
\end{equation}

We will show that the same formulas hold for $m-1$, hence establish
that they hold for all $0\le m\le M+1$.

First, for $\mu\in I_{m-1}$ and $\lambda\in\kappa_{F}$, consider
the formula for $C_{0,\mu+\varpi^{m-1}[\lambda]}^{m}$ , see (\ref{eq:formula for the coefficients}).
By (\ref{eq:contribution of outer sums}) with $l=d$, using the fact
that $[\lambda]^{q}=[\lambda]$ for all $\lambda\in\kappa_{F}$, we
know that
\[
\sum_{\lambda'\in\kappa_{F}}c_{k,\mu+\varpi^{m-1}[\lambda]+\varpi^{m}[\lambda']}^{m+1}[\lambda']^{dq}=\sum_{\lambda'\in\kappa_{F}}c_{k,\mu+\varpi^{m-1}[\lambda]+\varpi^{m}[\lambda']}^{m+1}[\lambda']^{d}\in\frac{\varpi^{n}}{a}\mathcal{O}_{C}
\]

which is the first summand in the first sum in (\ref{eq:formula for the coefficients})
with $j=0$.

For $i\le k-d$, since we assume $c_{i,\mu+\varpi^{m-1}[\lambda]+\varpi^{m}[\lambda']}^{m+1}\in\frac{\varpi^{n-d}}{a}\mathcal{O}_{C}$,
we see that
\[
\varpi^{k-i}\cdot\sum_{\lambda'\in\kappa_{F}}c_{i,\mu+\varpi^{m-1}[\lambda]+\varpi^{m}[\lambda']}^{m+1}[\lambda']^{i}\in\varpi^{d}\cdot\frac{\varpi^{n-d}}{a}\cdot\mathcal{O}_{C}=\frac{\varpi^{n}}{a}\cdot\mathcal{O}_{C}
\]
Also, for $k-d<i<k$, since we assume $c_{i,\mu+\varpi^{m-1}[\lambda]+\varpi^{m}[\lambda']}^{m+1}\in\frac{\varpi^{n}}{a}\mathcal{O}_{C}$,
we get
\[
\varpi^{k-i}\cdot\sum_{\lambda'\in\kappa_{F}}c_{i,\mu+\varpi^{m-1}[\lambda]+\varpi^{m}[\lambda']}^{m+1}[\lambda']^{i}\in\frac{\varpi^{n}}{a}\cdot\mathcal{O}_{C}
\]
This shows that the entire first sum in (\ref{eq:formula for the coefficients})
with $j=0$ lies in $\frac{\varpi^{n}}{a}\cdot\mathcal{O}_{C}$. In
addition, we have assumed that $c_{0,\mu+\varpi^{m-1}[\lambda]}^{m}\in\frac{\varpi^{n}}{a}\cdot\mathcal{O}_{C}$.
Therefore
\[
\sum_{i=0}^{k}c_{i,\mu}^{m-1}[\lambda]^{i}\in\frac{\varpi^{n}}{a}\mathcal{O}_{C}
\]

Next, we consider the formulas for $C_{j,\mu+\varpi^{m-1}[\lambda]}^{m}$
with $1\le j\le d$. By (\ref{eq:contribution of outer sums}) with
$l=d-j$ for $j\ne d$ and $l=q-1$ for $j=d$, using the fact that
$[\lambda]^{q}=[\lambda]$ for all $\lambda\in\kappa_{F}$, we know
that
\[
{k \choose j}\sum_{\lambda'\in\kappa_{F}}c_{k,\mu+\varpi^{m-1}[\lambda]+\varpi^{m}[\lambda']}^{m+1}[\lambda']^{dq-j}=
\]
\[
={k \choose j}\sum_{\lambda'\in\kappa_{F}}c_{k,\mu+\varpi^{m-1}[\lambda]+\varpi^{m}[\lambda']}^{m+1}[\lambda']^{l}\in\frac{\varpi^{n}}{a}\mathcal{O}_{C}\subseteq\varpi^{n-d}\mathcal{O}_{C}
\]
which is the first summand in the first sum in (\ref{eq:formula for the coefficients}).

Since for all $i$, we have $c_{i,\mu+\varpi^{m-1}[\lambda]+\varpi^{m}[\lambda']}^{m+1}\in\frac{\varpi^{n-d}}{a}\mathcal{O}_{C}$,
when considering $i<k$ we also have $1\le k-i$, hence
\[
\varpi^{k-i}{i \choose j}\sum_{\lambda'\in\kappa_{F}}c_{i,\mu+\varpi^{m-1}[\lambda]+\varpi^{m}[\lambda']}^{m+1}[\lambda']^{i-j}\in\varpi\cdot\frac{\varpi^{n-d}}{a}\mathcal{O}_{C}\subseteq\varpi^{n-d}\mathcal{O}_{C}
\]
where the last inclusion holds as $v_{F}(a)\le1$. This shows that
the entire first sum in (\ref{eq:formula for the coefficients}) lies
in $\varpi^{n-d}\mathcal{O}_{C}$.

Since we also have $c_{j,\mu+\varpi^{m-1}[\lambda]}^{m}\in\frac{\varpi^{n-d}}{a}\mathcal{O}_{C}$,
by (\ref{eq:formula for the coefficients}) we see that for all $1\le j\le d$
\[
\sum_{i=j}^{k}{i \choose j}c_{i,\mu}^{m-1}[\lambda]^{i-j}\in\varpi^{n-d-j}\mathcal{O}_{C}
\]
Therefore, by lemma \ref{subsec:Lemma divisibility of polynomial coefficients}
we have $c_{i,\mu}^{m-1}\in\varpi^{n-2d}\mathcal{O}_{C}$ for all
$i$.

Let $0<j\le d$. Looking at the formula for $C_{k-j,\mu}^{m}$, using
the fact that $k-j\ge dq-d=d(q-1)\ge2d$ (recall $q\ne2$), we see
that the second sum satisfies
\[
\varpi^{k-j}\sum_{i=k-j}^{k}{i \choose k-j}c_{i,[\mu]_{m-1}}^{m-1}[\lambda_{\mu}]^{i-(k-j)}\in\varpi^{2d}\cdot\varpi^{n-2d}\mathcal{O}_{C}=\varpi^{n}\mathcal{O}_{C}
\]
Also, we deduce from the hypothesis (\ref{eq:contribution of outer sums})
with $l=j$ that
\[
{k \choose k-j}\sum_{\lambda\in\kappa_{F}}c_{k,\mu+\varpi^{m}[\lambda]}^{m+1}[\lambda]^{k-(k-j)}=
\]
\[
={k \choose k-j}\sum_{\lambda\in\kappa_{F}}c_{k,\mu+\varpi^{m}[\lambda]}^{m+1}[\lambda]^{j}\in{k \choose k-j}\cdot\frac{\varpi^{n}}{a}\mathcal{O}_{C}\subseteq\varpi^{n}\mathcal{O}_{C}
\]

since $p\mid{k \choose k-j}={dq \choose (d-1)q+(q-j)}$ by Kummer's
theorem, and $v_{F}(a)\le1$.

For $i<k-d$, since we assume $c_{i,\mu+\varpi^{m}[\lambda]}^{m+1}\in\frac{\varpi^{n-d}}{a}\mathcal{O}_{C}$,
we see that
\[
\varpi^{k-i}{i \choose j}\cdot\sum_{\lambda\in\kappa_{F}}c_{i,\mu+\varpi^{m}[\lambda]}^{m+1}[\lambda]^{i-j}\in\varpi^{d+1}\cdot\frac{\varpi^{n-d}}{a}\cdot\mathcal{O}_{C}=\frac{\varpi^{n+1}}{a}\cdot\mathcal{O}_{C}\subseteq\varpi^{n}\mathcal{O}_{C}
\]
Also, for $k-d\le i<k$, since we assume $c_{i,\mu+\varpi^{m}[\lambda]}^{m+1}\in\frac{\varpi^{n}}{a}\mathcal{O}_{C}$,
and $1\le k-i$, we get
\[
\varpi^{k-i}{i \choose j}\cdot\sum_{\lambda\in\kappa_{F}}c_{i,\mu+\varpi^{m}[\lambda]}^{m+1}[\lambda]^{i-j}\in\varpi\cdot\frac{\varpi^{n}}{a}\cdot\mathcal{O}_{C}=\varpi^{n}\mathcal{O}_{C}
\]
This shows that both sums in (\ref{eq:formula for the coefficients})
lie in $\varpi^{n}\mathcal{O}_{C}$, hence also
\[
a\cdot c_{k-j,\mu}^{m}\in\varpi^{n}\mathcal{O}_{C}
\]

Furthermore, for any $1\le j\le d$, looking at the formulas for 
$$C_{j+(q-1),\mu}^{m},\ldots,C_{j+l(q-1),\mu}^{m},\ldots,$$
as $j+l(q-1)\ge j+q-1\ge q>2d$, by the same reasoning, we deduce
from the hypothesis (\ref{eq:contribution of outer sums}) with $l=d-j$
that
\[
a\cdot c_{j+l(q-1),\mu}^{m}\in\frac{\varpi^{n}}{a}\mathcal{O}_{C}
\]
Since $\sum_{l=0}^{\left\lfloor \frac{k-j}{q-1}\right\rfloor }c_{j+l(q-1),\mu}^{m}\in\frac{\varpi^{n}}{a}\mathcal{O}_{C}$
for $1\le j\le d$, this shows that $a\cdot c_{j,\mu}^{m}\in\frac{\varpi^{n}}{a}\mathcal{O}_{C}$
for all $0\le j\le k$.

We also note that $p\mid{dq \choose d+l(q-1)}$ for all $1\le l<d$,
by Kummer's theorem, therefore showing that
\[
a\cdot c_{d+l(q-1),\mu}^{m}\in\varpi^{n}\mathcal{O}_{C}
\]
Since $\sum_{l=0}^{d}c_{d+l(q-1),\mu}^{m}\in\frac{\varpi^{n}}{a}\mathcal{O}_{C}$,
we deduce that
\begin{equation}
c_{d,\mu}^{m}+c_{dq,\mu}^{m}\in\frac{\varpi^{n}}{a}\mathcal{O}_{C}\label{eq:d and dq}
\end{equation}
Therefore, we have established that
\[
c_{k,\mu}^{m}\in\frac{\varpi^{n}}{a^{2}}\mathcal{O}_{C},\quad c_{k-j,\mu}^{m}\in\frac{\varpi^{n}}{a}\mathcal{O}_{C}\quad\forall0<j\le d,,\quad c_{i,\mu}^{m}\in\frac{\varpi^{n-d}}{a}\mathcal{O}_{C}\quad\forall0\le i\le k
\]
Returning to the formulas for $C_{0,\mu}^{m},C_{1,\mu}^{m},\ldots,C_{d,\mu}^{m}$,
we see that for all $\lambda\in\mathbb{F}_{q}$ one has
\begin{equation}
\begin{split}
&\sum_{i=0}^{k}c_{i,\mu}^{m-1}[\lambda]^{i}\in\frac{\varpi^{n}}{a}\mathcal{O}_{C}, \\
&\sum_{i=1}^{k}ic_{i,\mu}^{m-1}[\lambda]^{i-1}\in\frac{\varpi^{n-1}}{a}\mathcal{O}_{C}, \\
\vdots& \\
&\sum_{i=d}^{k}{i \choose d}c_{i,\mu}^{m-1}[\lambda]^{i-d}\in\frac{\varpi^{n-d}}{a}\mathcal{O}_{C}
\end{split}
\end{equation}
Therefore, by Lemma \ref{subsec:Lemma divisibility of polynomial coefficients}
we have $c_{i,\mu}^{m-1}\in\frac{\varpi^{n-d}}{a}\mathcal{O}_{C}$
for all $i$. Moreover, we see that
\begin{equation}
\begin{split}
&c_{0,\mu}^{m-1}\in\frac{\varpi^{n}}{a}\mathcal{O}_{C},\quad\sum_{l=0}^{\left\lfloor \frac{k-j}{q-1}\right\rfloor }c_{j+l(q-1),\mu}^{m}\in\frac{\varpi^{n}}{a}\mathcal{O}_{C}\quad\forall1\le j\le d \\
&c_{i,\mu}^{m-1}\in\frac{\varpi^{n-d}}{a}\mathcal{O}_{C}\quad\forall0\le i\le k
\end{split}
\end{equation}
It remains to establish (\ref{eq:contribution of outer sums}) for
$m$. Looking at the equation for $C_{d,\mu}^{m}$, we see that for
all $\mu$ we have
\[
\varpi^{d}\cdot\sum_{i=d}^{k}{i \choose d}c_{i,[\mu]_{m-1}}^{m-1}[\lambda_{\mu}]^{i-d}-a\cdot c_{d,\mu}^{m}\in\varpi^{n}\mathcal{O}_{C}
\]
since $p\mid{k \choose d}={dq \choose d}$. Fixing $\mu\in I_{m-1}$
and summing over all $\lambda\in\mathbb{F}_{q}$, we get
\[
a\cdot\sum_{\lambda\in\mathbb{F}_{q}}c_{d,\mu+\varpi^{m-1}[\lambda]}^{m}[\lambda]^{l}-\varpi^{d}\cdot\sum_{i=d}^{k}{i \choose d}c_{i,\mu}^{m-1}\sum_{\lambda\in\mathbb{F}_{q}}[\lambda]^{i+l-d}\in\varpi^{n}\mathcal{O}_{C}
\]
for any $l\in\{0,1,2,\ldots,d,q-1\}$.

However, as
\[
\sum_{\lambda\in\mathbb{F}_{q}}[\lambda]^{i}\equiv\begin{cases}
-1 & q-1\mid i,\qquad i\ne0\\
0 & else
\end{cases}\mod p
\]
we obtain
\[
a\cdot\sum_{\lambda\in\mathbb{F}_{q}}c_{d,\mu+\varpi^{m-1}[\lambda]}^{m}[\lambda]^{l}-\varpi^{d}\cdot\sum_{h=1}^{\left[\frac{k-d+l}{q-1}\right]}{d-l+h(q-1) \choose d}c_{d-l+h(q-1),\mu}^{m-1}\in\varpi^{n}\mathcal{O}_{C}
\]
Fix some $l\in\{0,1,\ldots,d\}$. Note that for all $h\le d-l$, one
has
\[
p\mid{d-l+h(q-1) \choose d}={h\cdot q+(d-l-h) \choose d}
\]
This means we have
\begin{equation}
a\cdot\sum_{\lambda\in\mathbb{F}_{q}}c_{d,\mu+\varpi^{m-1}[\lambda]}^{m}[\lambda]^{l}-\varpi^{d}\cdot\sum_{h=d-l+1}^{d}{d-l+h(q-1) \choose d}c_{d-l+h(q-1),\mu}^{m-1}\in\varpi^{n}\mathcal{O}_{C}\label{eq:gather from all lambdas}
\end{equation}
For $l=0$, this already implies
\[
\sum_{\lambda\in\mathbb{F}_{q}}c_{d,\mu+\varpi^{m-1}[\lambda]}^{m}\in\frac{\varpi^{n}}{a}\mathcal{O}_{C}
\]
hence by (\ref{eq:d and dq})
\[
\sum_{\lambda\in\mathbb{F}_{q}}c_{dq,\mu+\varpi^{m-1}[\lambda]}^{m}\in\frac{\varpi^{n}}{a}\mathcal{O}_{C}
\]
For arbitary $l$, we proceed as follows.

Consider the formulas for $C_{0,\mu}^{m-1},C_{1,\mu}^{m-1},\ldots,C_{d,\mu}^{m-1}$.
We obtain as before that $c_{i,\mu}^{m-2}\in\varpi^{n-2d}\mathcal{O}_{C}$
for all $i$.

We may now consider the formulas for $C_{dq-l,\mu}^{m-1},C_{(d-1)q-l+1,\mu}^{m-1},\ldots,C_{(d-l+1)q-1,\mu}^{m-1}$.
Since
\[
(d-l+1)q-1\ge q-1\ge2d
\]
we get
\[
{dq \choose d-l+h(q-1)}\cdot\sum_{\lambda\in\mathbb{F}_{q}}c_{dq,\mu+\varpi^{m-1}[\lambda]}^{m}[\lambda]^{l}+a\cdot c_{d-l+h(q-1),\mu}^{m-1}\in\varpi^{n}\mathcal{O}_{C}
\]
for all $d-l+1\le h\le d$. Substituing back in (\ref{eq:gather from all lambdas}),
we get

\begin{equation}
A \cdot \sum_{\lambda\in\mathbb{F}_{q}}c_{dq,\mu+\varpi^{m}[\lambda]}^{m}[\lambda]^{l}\in\varpi^{n}\mathcal{O}_{C}
\end{equation}
where 
\begin{equation}
A = \left(a+\frac{1}{a}\sum_{h=d-l+1}^{d}\varpi^{d}\cdot{dq \choose d-l+h(q-1)}\cdot{d-l+h(q-1) \choose d}\right)
\end{equation}

Note that $p\mid{dq \choose d-l+h(q-1)}={dq \choose (h-1)q+q+d-l-h}$,
hence
\[
v_{F}\left(\frac{\varpi^{d}}{a}\cdot{dq \choose d-l+h(q-1)}\cdot{d-l+h(q-1) \choose d}\right)\ge d+1-v_{F}(a)
\]
But, as $v_{F}(a)\le1<\frac{1+d}{2}$, it follows that $v_{F}(a)<d+1-v_{F}(a)$,
so that we must have
\[
a\cdot\sum_{\lambda\in\mathbb{F}_{q}}c_{dq,\mu+\varpi^{m}[\lambda]}^{m}[\lambda]^{l}\in\varpi^{n}\mathcal{O}_{C}
\]
as claimed.

Finally, looking at the formulas for $C_{dq}^{m-1},\ldots,C_{d+q-1}^{m-1}$,
we have
\[
{dq \choose d+h(q-1)}\cdot\sum_{\lambda\in\mathbb{F}_{q}}c_{dq,\mu+\varpi^{m-1}[\lambda]}^{m}[\lambda]^{q-1}+a\cdot c_{d+h(q-1),\mu}^{m-1}\in\varpi^{n}\mathcal{O}_{C}
\]
for all $1\le h\le d-1$, and
\[
\sum_{\lambda\in\mathbb{F}_{q}}c_{dq,\mu+\varpi^{m-1}[\lambda]}^{m}+a\cdot c_{dq,\mu}^{m-1}\in\varpi^{n}\mathcal{O}_{C}
\]
Substituting in (\ref{eq:gather from all lambdas}), and recalling
that $\sum_{h=0}^{d}c_{d+h(q-1)}^{m-1}\in\frac{\varpi^{n}}{a}\mathcal{O}_{C}$,
we obtain
\[
\left(a+\frac{1}{a}\cdot\varpi^{d}\sum_{h=0}^{d-1}{d+h(q-1) \choose d}{dq \choose d+h(q-1)}\right) \sum_{\lambda\in\mathbb{F}_{q}}c_{dq,\mu+\varpi^{m-1}[\lambda]}^{m}[\lambda]^{l}\in\varpi^{n}\mathcal{O}_{C}
\]
since $\varpi^{2}\mid\varpi^{d}\cdot{dq \choose d}$.

Since $v_{F}(a)\le1<\frac{d+1}{2}$, this is only possible if $\sum_{\lambda\in\mathbb{F}_{q}}c_{dq,\mu+\varpi^{m-1}[\lambda]}^{m}[\lambda]^{q-1}\in\frac{\varpi^{n}}{a}\mathcal{O}_{C}$.
Therefore, we are done, and $\epsilon=d+v_{F}(a)$ suffices.
\end{proof}

\newpage

\bibliographystyle{amsalpha}
\bibliography{arxivinv}

\end{document}